\documentclass[twoside,11pt]{amsart}
\usepackage[all]{xy}
\CompileMatrices
\usepackage{amsfonts}
\usepackage{amssymb}
\usepackage{amsthm}
\usepackage{amsmath}
\usepackage{ifthen}
\usepackage{url}
\usepackage[backref]{hyperref}
\usepackage[usenames]{color}
\usepackage{comment}
\usepackage{color}
\usepackage{framed}
\definecolor{shadecolor}{gray}{0.875}
\specialcomment{personal}{\begin{shaded}}{\end{shaded}}
\excludecomment{personal}

 \newtheorem{thm}{Theorem}[section]
 \newtheorem{prop}[thm]{Proposition}
 
 \newtheorem{lem}[thm]{Lemma}

\newtheorem{rem}[thm]{Remark}

\theoremstyle{definition}
\newtheorem{defn}[thm]{Definition}

\newtheorem{ex}[thm]{Example}

\theoremstyle{remark}

\font\russ=wncyr10  1
\def\sha{\hbox{\russ\char88}}

\renewcommand{\1}{1\hspace{-.08cm}\mathrm{I}}
\renewcommand{\L}{\ifmmode {\mathcal{L}}\else$\mathcal{L}$\ \fi}

\renewcommand{\d}{\mathbf{d}}
\renewcommand{\u}{\mathbf{1}}
\newcommand{\id}{\mathrm{id}}
\newcommand{\bbC}{\ifmmode {\mathbb{C}}\else$\mathbb{C}$\ \fi}
\newcommand{\bbR}{\ifmmode {\mathbb{R}}\else$\mathbb{R}$\ \fi}

\renewcommand{\r}{\mathrm{R}\Gamma}

\newcommand{\be}{\begin{equation}}
\newcommand{\ee}{\end{equation}}

\newcommand{\qnr}{ \widehat{\mathbb{Q}_p^{ur}}}
\newcommand{\zpnr}{ \widehat{\mathbb{Z}_p^{ur}}}
\newcommand{\fpbar}{\ifmmode {\overline{\mathbb{F}_p}}\else$\mathbb{F}_p$\ \fi}

\newcommand{\fp}{\ifmmode {\mathbb{F}_p}\else$\mathbb{F}_p$\ \fi}
\newcommand{\zp}{\ifmmode \mathbb{Z}_p\else$\mathbb{Z}_p$\ \fi}
\newcommand{\zpur}{\ifmmode \widehat{\zp^{ur}}\else $\widehat{\zp^{ur}}$\ \fi}
\newcommand{\Tun}{\ifmmode \mathbb{T}_{un}\else$\mathbb{T}_{un}$\ \fi}

\newcommand{\z}{\mathbb{Z}}
\newcommand{\Z}{\mathbb{Z}}
\newcommand{\zpMod}{\ifmmode\mbox{$\zp$-Mod}\else$\zp$-Mod \fi}
\newcommand{\Mod}{\ifmmode\mbox{$\Lambda$-Mod}\else$\Lambda$-Mod \fi}
\renewcommand{\mod}{\ifmmode\mbox{$\Lambda$-mod}\else$\Lambda$-mod
\fi}
\newcommand{\La}{\ifmmode\Lambda\else$\Lambda$\fi}

\newcommand{\Hom}{{\mathrm{Hom}}}
\newcommand{\Ext}{{\mathrm{Ext}}}

\renewcommand{\H}{\mathrm{H}}

\newcommand{\M}{\ifmmode {\frak M}\else${\frak M}$ \fi}
\newcommand{\m}{\ifmmode {\frak m}\else$\frak m$ \fi}
\newcommand{\mh}{\ifmmode {\frak m}(H)\else${\frak m}(H)$ \fi}
\newcommand{\p}{\ifmmode {\frak p}\else${\frak p}$\ \fi}
\renewcommand{\P}{\ifmmode {\frak P}\else${\frak P}$\ \fi}
\newcommand{\e}{\ifmmode {\mathcal{E}}\else$\mathcal{E}$ \fi}

\newcommand{\C}{\mathcal{C}}

\newcommand{\T}{\mathbb{ T}}
\renewcommand{\O}{\mathcal{ O}}
\newcommand{\G}{\ifmmode {\mathcal{G}}\else${\mathcal{G}}$\ \fi}
\newcommand{\A}{\ifmmode {\mathcal{A}}\else${\mathcal{ A}}$\ \fi}
\newcommand{\lu}{\mathbb{U}} 

\renewcommand{\projlim}[1] {{\lim\limits_{\stackrel{\displaystyle
\longleftarrow}{#1}}}}

\newcommand{\kl}{[\![}
\newcommand{\kr}{]\!]}
\newcommand{\Qp}{\ifmmode {{\Bbb Q}_p}\else${\Bbb Q}_p$\ \fi}
\newcommand{\qp}{\ifmmode {\Bbb Q}_p\else${\Bbb Q}_p$\ \fi}
\newcommand{\ql}{\ifmmode {{\Bbb Q}_l}\else${\Bbb Q}_l$\ \fi}
\newcommand{\Q}{\ifmmode {\Bbb Q}\else${\Bbb Q}$\ \fi}
\newcommand{\q}{\ifmmode {\Bbb Q}\else${\Bbb Q}$\ \fi}

\newcommand{\Ind}{\mathrm{Ind}}

\newcommand{\Co}{\mbox{$\Bbb C$}}

\def\sectionnam{\@empty}
\def\subsectionnam{\@empty}

\begin{document}


\title[On Kato's local $\epsilon$-isomorphism Conjecture]{On Kato's local $\epsilon$-isomorphism Conjecture for rank one Iwasawa modules}%

\author{Otmar Venjakob}%
\address{Universit\"{a}t Heidelberg\\ Mathematisches Institut\\
Im Neuenheimer Feld 288\\ 69120 Heidelberg, Germany.} \email{venjakob@mathi.uni-heidelberg.de}
\urladdr{http://www.mathi.uni-heidelberg.de/\textasciitilde otmar/}
\thanks{I acknowledge support  by the ERC and DFG}



\date{\today}%
\maketitle
\thispagestyle{empty}

\begin{abstract}
 This paper contains a complete proof of Fukaya's and Kato's $\epsilon$-iso\-morphism conjecture  in \cite{fukaya-kato} for invertible $\Lambda$-modules (the case of $V = V_0(r)$ where $V_0$ is
unramified of dimension $1$). Our results rely heavily on Kato's unpublished proof of (commutative) $\epsilon$-isomorphisms for one
dimensional representations of $G_{\qp}$ in \cite{kato-lnmII}, but apart from fixing some sign-ambiguities in (loc.\ cit.)  we use the theory of $(\phi,\Gamma)$-modules instead of syntomic cohomology. Also, for the convenience of the reader we give a slight modification or rather reformulation of it in the language of \cite{fukaya-kato} and extend it   to the (slightly non-commutative) semi-global setting. Finally we discuss some direct applications  concerning the Iwasawa theory of CM elliptic curves, in particular  the local Iwasawa   Main
Conjecture  for CM elliptic curves $E$ over the   extension of $\qp$ which trivialises the $p$-power division points $E(p)$ of $E$. In this sense the paper is   complimentary to the joint work \cite{bou-ven} on noncommutative Main Conjectures for CM elliptic curves.
\end{abstract}

\section{Introduction}

The significance of (local) $\epsilon$-factors \`{a} la Deligne and Tate or more general of the
(conjectural) $\epsilon$-isomorphism suggested by Fukaya and Kato in \cite[\S 3]{fukaya-kato} is at
least twofold: First of all they are important ingredients to obtain a precise functional equation
for $L$-functions or more generally for (conjectural) $\zeta$-isomorphism (loc.~ cit., \S 2) of
motives in the context of equivariant or non-commutative Tamagawa number conjectures, see e.g.\ Theorem \ref{functionalequation}; secondly they
are essential in interpolation formulae of (actual) $p$-adic $L$-functions and for the relation
between $\zeta$-isomorphisms and (conjectural, not necessarily commutative) $p$-adic $L$-functions
as discussed in (loc.~cit., \S 4). Of course the two occurrences are closely related, for a survey
on these ideas see also \cite{ven-BSD}.

Our motivation for writing this article stems from
 one of the main results, theorem 8.4, of  \cite{burns-ven2} (see Theorem \ref{burns-descent})
describing under which conditions the validity of a (non-commutative) Iwasawa main conjecture for a
critical (ordinary at $p$) motive $M$ over some $p$-adic Lie extension $F_\infty$ of $\Q$ implies
parts of the Equivariant Tamagawa Number Conjecture (ETNC) by Burns and Flach for $M$ with respect
to a finite Galois extension $F\subseteq F_\infty$ of $\Q$. Due to the second above mentioned meaning it requires among others the existence of
an $\epsilon$-isomorphism
\begin{equation}
\label{appCM}
 \epsilon_{p,\zp[G(F/\Q)]}(\hat{\mathbb{T}}_F) :
\u_{\zp[G_{F/\Q}]}\to
\d_{\zp[G_{F/\Q}]}(\r(\qp,\hat{\mathbb{T}}_F))\d_{\zp[G_{F/\Q}]}(\hat{\mathbb{T}}_F)\end{equation}
in the sense of \cite[Conj.\ 3.4.3]{fukaya-kato}, where the Iwasawa module $\hat{\mathbb{T}}_F$ is
related to the ordinary condition of $M,$ e.g.\ for an (ordinary) elliptic curve $E$ it arises from
the formal group part of the usual Tate module of $E.$ Unfortunately, very little is known about
the existence of such $\epsilon$-isomorphism in general. To the knowledge of the author it is not
even contained in the literature for   $\hat{\mathbb{T}}_F$ attached to a $CM$-elliptic curve  $E$ and
the trivialising extension $F_\infty:=F(E(p))$, where $E(p)$ denotes group of $p$-power division points of $E$. Well, in principle a rough sketch of a proof is
contained in Kato's work \cite{kato-lnmII}, which unfortunately has never been published sofar. Moreover there were still some sign-ambiguities which we fix in this paper, in particular it turns out that one has to take $-\mathcal{L}_{K,\epsilon^{-1}}$, i.e., $-1$ times the classical Coleman map \eqref{col2}, in the construction of the epsilon isomorphism \eqref{epsTun}.

Recently, Benois and Berger \cite{benois-berger} have proved the conjecture $C_{EP}(L/K,V)$ for
arbitrary crystalline representations $V$ of $G_K$, where $K$ is an unramified extension of $\qp$
and $L$ a finite subextension of $K_\infty=K(\mu(p))$ over $K.$ Although they mention in their
introduction ``Les m\^eme arguments, avec un peu plus de calculs, permettent de d\'{e}montrer la
conjecture $C_{EP}(L/K,V)$ pour toute extension $L/K$ contenue dans $\mathbb{Q}_p^{ab}$. Cette
petite g\'{e}n\'{e}ralisation est importante pour la version \'equivariante des conjectures de Bloch et
Kato'', they leave it as an ``exercise'' to the reader. In the special case $V=\qp(r),$ $r\in\z$,
Burns and Flach \cite{bf4} prove a local ETNC using global ingredients in a semi-local setting, while in the
above example we need it for $V=\qp(\eta)(r)$, where $\eta$ denotes an unramified character. Also
we would like to stress that the existence of the $\epsilon$-isomorphisms \`{a} la Fukaya and Kato is a slightly
finer statement then the $C_{EP}(L/K,V)$-conjecture or the result of Burns and Flach, because the
former one states that a certain family of certain precisely defined units of integral group
algebras of finite groups in a certain tower can be interpolated by a unit in the corresponding
Iwasawa algebra while in the latter ones ``only'' a family  of lattices is ``interpolated'' by one over
the Iwasawa algebra.

The aim of this article, which also might hopefully serve as a survey into the subject, is to
provide detailed and complete arguments for the existence of the $\epsilon$-isomorphism \[\epsilon_\Lambda(\T(T)):\u_{\widetilde{\Lambda}} \to\d_\Lambda(R\Gamma(\qp,\T(T) ))_{\widetilde{\Lambda}}\d_\Lambda(\T(T) )_{\widetilde{\Lambda}}\] where
$\Lambda=\Lambda(G)$ is the Iwasawa algebra of $G=G(K_\infty/\qp)$ for any (possibly infinite)
unramified extension $K$ of $\qp$,   $T=\zp(\eta)(r)$ and $ R\Gamma(\qp,\T(T) )$ denotes the complex calculating local Galois cohomology of $\T(T)$, the usual Iwasawa theoretic deformation of $T$ (see \eqref{defomation}). Furthermore, for an associative ring $R$ with one, $\d_R$ denotes the determinant functor with  $\u_R=\d_R(0)$ (see Appendix \ref{determinants})  while $\widetilde{\Lambda}$ is defined in \eqref{tilde}. We are mainly interested in the case,
where $G\cong \mathbb{Z}_p^2\times\Delta$ for a finite group $\Delta$ - such extensions arise for example by adjoining the $p$-power division points of a CM elliptic curve to the base field as above. This corresponds to a (generalised) conjecture $C_{IW}(K_\infty/\qp)$ (in the notation of Benois and Berger) originally
due to Perrin-Riou. It is the first example of an $\epsilon$-isomorphism associated with a two
dimensional $p$-adic Lie group extension. Following Kato's approach we construct a {\em universal}
$\epsilon$-isomorphism $\epsilon_\Lambda(\T(\zp(1)))$, from which all the others arise by suitable
twists and descent. But while Kato constructs it first over      cyclotomic $\zp$-extensions and then takes limits, here we construct it directly  over $(\mathbb{Z}_p^2\times\Delta)$-extensions (and then take limits). To show that they satisfy the right interpolation property with respect to
Artin(=Dirichlet) characters of $G,$ we use the theory of $(\phi,\Gamma)$-modules and Berger's
explicit formulae in \cite{berger-exp} instead of the much more involved  syntomic cohomology and
Kato's reciprocity laws for formal groups. In contrast to Kato's unpublished preprint, in which
 he uses the language of \'{e}tale sheaves and cohomology, we prefer Galois cohomology as used also in \cite{fukaya-kato}.
In order to work out in detail Kato's reduction argument in \cite{kato-lnmII} to the case of
trivial $\eta$ we have to show a certain twist
 compatibility  of Perrin-Riou's exponential map/Coleman map for $T$ versus $\zp(r)$ over a  trivialising
 extension $K_\infty$ for $\eta$, see Lemma \ref{twistL}. Going over to semi-local settings we obtain the
 first $\epsilon$-isomorphism over a (slightly)
non-commutative ring. In a forthcoming paper \cite{loefflerzerbesven}, using the techniques of
\cite{benois-berger} and \cite{zerbes-loeffler}, we  are going to extend these results to the
case of arbitrary crystalline representations for the same tower of local fields as above. Of
course it would be most desireable to extend the existence of $\epsilon$-isomorphism also to
non-abelian local extensions, but which seems to require completely new ideas and   to be out of reach
at present (see \cite{iz} for some examples). Some evidence in that direction has been provided by
Fukaya (unpublished).

Combined with Yasuda's work \cite{yasuda} concerning $\epsilon$-isomorphism for $l\neq p,$ we also
obtain in principal a purely local proof of the above mentioned result by Burns and Flach for $V=\qp(r)$.

{\em Acknowledgements:} I am grateful to Denis Benois and Laurent Berger for a   kind explanation
of their work in \cite{benois-berger}. Also I would like to thank Matthias Flach and Adebisi
Agboola for helpful discussions.    I am indebted to Dmitriy Izychev and Ulrich Schmitt for pointing out a couple of typos. Finally, I am grateful to the anonymous referee for valuable suggestions which helped to improve  the article.

\section{Kato's proof for one dimensional representations}\label{Katoproof}

Let $p$ be a  prime and $K$ be any  unramified (possibly infinite) Galois extension of $\qp.$
We set $K_n:=K(\mu_{p^n})$ for $0\leq n \leq \infty$ and \[\Gamma=G(\mathbb{Q}_{p,\infty}/\qp)\cong
\mathbb{Z}_p^\times.\] Recall that the maximal unramified extension
$\qp^{ur}$ and the maximal abelian extension $\qp^{ab}$ of $\qp$ are given as $\qp(\mu(p'))$ and
$\qp(\mu)=\qp^{ur}(\mu(p)),$   where $\mu(p)$ and $\mu(p')$ denote the $p$-primary and prime-to-$p$
part of $\mu,$ the group of all roots of unity, respectively. In particular, we have the canonical decomposition
\begin{eqnarray*}
G(\qp^{ab}/\qp)&=&G(\qp^{ur}/\qp)\times G(\mathbb{Q}_{p,\infty}/\qp)\\
&=& \hat{\z}\times \mathbb{Z}_p^\times,
\end{eqnarray*}
under which per definitionem $\tau_p$ corresponds to $(\phi, 1)$ (and by abuse of notation also to
its image in $G$ below), where $\phi:=Frob_p$ denotes the arithmetic Frobenius $x\mapsto x^p.$ We
put
\[H:=H_K:=G(K/\qp)=\overline{<\phi>}\]
and
\[G:=G(K_\infty/\qp)\cong H\times \Gamma.\]

Assume that $G$ is a $p$-adic Lie-group, i.e., $H$ is the product of a finite abelian group of
order prime to $p$ with a (not necessarily strict) quotient of $\zp.$ By
\[\Lambda:=\Lambda(G):=\zp\kl G \kr\]
 we denote as usual the Iwasawa algebra of $G.$ Also we write $\widehat{\zp^{ur}}$ for the ring of
Witt vectors $W(\overline{\fp})$ with its natural action by $\phi$ and we set
\begin{equation}
\label{tilde}
\widetilde{\Lambda}=\Lambda\widehat{\otimes}_{\zp}\zpur =\zpur \kl G \kr.  \end{equation}
By \[\Tun:=\Lambda^\sharp(1)\] we denote the free $\Lambda$-module of rank one with the following
Galois action
\[\chi_{un}:G_{\qp} \to \Lambda^\times,  \;\; \sigma\mapsto [\Tun,\sigma]:=\bar{\sigma}^{-1}\kappa(\sigma),\]
where $\bar{\phantom{m}}: G_{\qp}\twoheadrightarrow G$ is the natural projection map and
$\kappa:G_{\qp}\twoheadrightarrow\zp^\times$ is the $p$-cyclotomic character. Furthermore, we write
\[\lu(K_\infty):=\projlim{L,i} \mathcal{O}_L^\times/p^i\] for the $\Lambda$-module of local units,
where $L$ and $i$ run through the finite subextensions of $K_\infty/\qp$ and the natural numbers,
respectively, and the transition maps are induced by the norm. Finally we fix once and for all a
$\zp$-basis $\epsilon=(\epsilon_n)_n$ of $\zp(1)=\projlim{n}  \mu_{p^n}.$

Setting $\Lambda_a=\{x\in \widetilde{\Lambda}| (1\otimes\phi) (x)= (a\otimes 1)\cdot x\}$ for $a\in\Lambda^\times=
K_1(\Lambda)$ we obtain

\begin{prop}\label{twist}
For $a=[\Tun,\tau_p]^{-1}=\tau_p$ there is a canonical isomorphism \[  \Lambda_a \cong \left\{
        \begin{array}{ll}
          \O_K\kl \Gamma\kr, & \hbox{if $H$ is finite;} \\
          \projlim{\qp\subseteq K'\subseteq K \hbox{\footnotesize finite, }Tr} \O_{K'}\kl \Gamma \kr , & \hbox{if $H$ is infinite.}
        \end{array}
      \right.
\] as $\Lambda$-modules and all modules are free of rank one.
\end{prop}

\begin{proof}
We first assume $H=<\tau_p>$ to be finite of order $d$ and replace $\Gamma$ by a finite quotient
without changing the notation. Then any element $x\in \widetilde{\Lambda}=\zpur[\Gamma][H]$ can be
uniquely written as $\sum_{i=0}^{d-1} a_i \tau_p^i$ with $a_i\in \zpur[\Gamma]$ and $\phi$  acts
coefficient wise on the latter elements. The calculation
\begin{eqnarray*}
(1\otimes \phi)(x)-(\tau_p \otimes 1)x &=& \sum_{i=0}^{d-1} \phi(a_i) \tau_p^i
-\sum_{i=0}^{d-1}  a_i \tau_p^{i+1} \\
&=& \sum_{i=0}^{d-1}\big( \phi(a_i) -a_{i-1}\big) \tau_p^i
\end{eqnarray*}
with $a_{-1}:=a_{d-1}$ shows that $x$ belongs to $\Lambda_a$ if and only if $\phi^d(a_i)=a_i$ and
$\phi^{-i}(a_0)=a_i$ for all $i.$
 As $\zpur^{\phi^d=1}=\O_K,$ the canonical map
\[\Lambda_a\cong \O_K[\Gamma], \;\; \sum a_i \tau_p^i\mapsto a_0,\] is an isomorphism of
$\Lambda$-modules, the inverse of which is \[x\mapsto \sum_{h\in H}h\otimes h^{-1}(x)\] and which
is obviously functorial in $\Gamma,$ whence the same result follows for the original (infinite)
$\Gamma.$

Now, for    a surjection $\pi:H''\twoheadrightarrow H'$ it is easy to check that the trace
$Tr_{K''/K'}:\O_{K''}\to \O_{K'}$ induces a commutative diagram \be\label{twisttrace}\xymatrix{
  {\Lambda''_{a''}} \ar[d]_{\pi} \ar[r]^{\cong} & {\O_{K''}\kl \Gamma\kr} \ar[d]^{Tr_{K''/K'}} \\
   {\Lambda'_{a'}}  \ar[r]^{\cong} & {\O_{K'}\kl \Gamma\kr}   ,}\ee

whence the first claim follows. From the normal basis theorem for finite fields we obtain
(non-canonical) isomorphisms \[\O_{K'}\cong \zp[H_{K'}],\] which are compatible with trace and natural
projection maps. Indeed, the sets $S_{K'}:=\{a\in \O_{K'}| \zp[H_{K'}]a=\O_{K'}\} \cong \zp[H_{K'}]^\times$ are
compact, since $1 + Jac(\zp[H_{K'}])$ for the Jacobson radical $Jac(\zp[H_{K'}])$ is open in
$\zp[H_{K'}]^\times,$ and thus $\projlim{K'} S_{K'}$ is non-empty. Hence the trace maps induce
(non-canonical) isomorphisms $\projlim{K'} \O_{K'}\cong \zp\kl H\kr$ and $\projlim{K'} \O_{K'}\kl \Gamma\kr
\cong \zp\kl G\kr,$ respectively.
\end{proof}

We now review Coleman's exact sequence \cite{col1979,coleman83}, which is one   crucial ingredient
in the construction of the $\epsilon$-isomorphism. Let us first assume that {\em $K/\qp$ is
finite}.

Then $\lu(K_\infty):=\projlim{n,i} \mathcal{O}_{K_n}^\times/p^i$ with $K_n:=K(\mu_{p^n})$ and the following sequence of
$\Lambda$-modules is exact \be\label{colemanfinite}\xymatrix{
  0 \ar[r]^{ } & {\zp(1)}  \ar[r]^{\iota} & {\lu(K_\infty)} \ar[r]^{\mathrm{Col}} & {\mathcal{O}_K\kl \Gamma \kr}  \ar[r]^{\pi} & {\zp(1)} \ar[r]^{ } &
0,
}\ee
 where

\begin{itemize}
  \item  $\iota(\epsilon)=  \epsilon$,
  \item $\mathrm{Col}(u):=\mathrm{Col}_\epsilon(u)$ is defined by the rule \be\label{colpowerseries}\mathcal{L}(g_u):=(1-\frac{\varphi}{p})\log(g_u)=\frac{1}{p} \log(\frac{g_u^p}{ \varphi(g_u)})=\mathrm{Col}(u)\cdot
(X+1)\ee in $\O_K\kl X\kr$ with $ g_u:=g_{u,\epsilon}\in\O_K\kl X\kr$ the Coleman power series
satisfying $g^{\phi^{-n}}(\epsilon_n-1)=u_n$ for all $n.$ Here $\phi$ is acting coefficientwise on
$g_u=g_u(X)$, while $ \varphi:\O_K\kl X\kr \to \O_K\kl X\kr $ is induced by $X\mapsto (X+1)^p-1$ and
the action of $\phi$ on the coefficients. Furthermore, the $\O_K$-linear action of $\O_K\kl \Gamma
\kr$ on $\O_K\kl X \kr$ is induced by $\gamma\cdot X=(1+X)^{\kappa(\gamma)}-1.$
  \item $\pi$ is the composite of $  {\O_K\kl\Gamma\kr  }  \to { \O_K},$ $\gamma\mapsto \kappa(\gamma),$  followed by the trace
  $ Tr_{K/\qp}: { \O_K} \to {\zp }   $ (and strictly speaking followed by $\zp\to \zp(1),\; c\mapsto c \epsilon$).
\end{itemize}

Using Proposition \ref{twist} and the isomorphism
\[ \Lambda_{[\Tun,\tau_p]^{-1}}\cong \Tun\otimes_\Lambda \Lambda_{[\Tun,\tau_p]^{-1}}, a \mapsto (1\otimes  \epsilon) \otimes a,\]
we thus obtain an exact sequence of $\Lambda$-modules \be\label{col2} \xymatrix{
  0 \ar[r]^{ } & {\zp(1)}  \ar[r]^{ } & {\lu(K_\infty)}  \ar[r]^(0.3){\mathcal{L}_{K,\epsilon}} & {\Tun(K_\infty)\otimes_\Lambda \Lambda_{[\Tun,\tau_p]^{-1}}} \ar[r]^{ } & {\zp(1)} \ar[r]^{ } & 0   }.
\ee
In the end we actually shall need the  analogous exact sequence
\be\label{col2bis} \xymatrix{
  0 \ar[r]^{ } & {\zp(1)}  \ar[r]^{ } & {\lu(K_\infty)}  \ar[r]^(0.3){-\mathcal{L}_{K,-\epsilon}} & {\Tun(K_\infty)\otimes_\Lambda \Lambda_{[\Tun,\tau_p]^{-1}}} \ar[r]^{ } & {\zp(1)} \ar[r]^{ } & 0   }.
\ee
 where we replace $\epsilon$ by $-\epsilon$ everywhere in the construction and where we multiply (only) the middle map by $-1$. Note that the maps involving $\zp(1)$ do not change compared with \eqref{col2}.

In order to deal with the case that $K/\qp$ {\em is infinite,} \i.e., $p^\infty |[K:\qp],$ let
$\qp\subseteq L\subseteq L' \subseteq K$ be finite intermediate extensions. We claim that the
following diagram
\be \label{commCol}
\xymatrix{
   0 \ar[r]^{ } & {\zp(1)}  \ar[r]^{ } \ar[d]_{N_{L_\infty'/L_\infty}=[L':L]\cdot}   & {\lu(L_\infty)}  \ar[r]^(0.3){\mathcal{L}_{L',\epsilon}} \ar[d]_{N_{L_\infty'/L_\infty}}   & {\Tun(L_\infty')\otimes_\Lambda \Lambda_{[\Tun,\tau_p]^{-1}}} \ar[r]^{ }  \ar[d]_{pr_{L'/L}}
  \ar[r]^{ } & {\zp(1)}\ar[d]_{=} \ar[r]^{ } & 0   \\
  0 \ar[r]^{ } & {\zp(1)}  \ar[r]^{ } & {\lu(L_\infty)}  \ar[r]^(0.3){\mathcal{L}_{L,\epsilon}} & {\Tun(L_\infty)\otimes_\Lambda \Lambda_{[\Tun,\tau_p]^{-1}}} \ar[r]^{ } & {\zp(1)} \ar[r]^{ } & 0   }
\ee commutes, where the norm maps $N_{L_\infty'/L_\infty}=N_{L'/L}$ are induced  by $N_{L_n'/L_n}$
for all $n,$ which on $\zp(1)$ amounts to multiplication by $[L':L]$ while
$N_{L_\infty'/L_\infty}:\lu(L'_\infty)\to \lu(L_\infty)$ is nothing else than the projection on the
corresponding   inverse (sub)system. Recalling \eqref{twisttrace} this is equivalent to the
commutativity of
\be \label{commCol1} \xymatrix{
   0 \ar[r]^{ } & {\zp(1)}  \ar[r]^{ } \ar[d]_{N_{L_\infty'/L_\infty}=[L':L]\cdot}   & {\lu(L'_\infty)}  \ar[r]^{\mathrm{Col}_{L',\epsilon}} \ar[d]_{N_{L_\infty'/L_\infty}}   & {\O_{L'}\kl \Gamma \kr} \ar[r]^{ }  \ar[d]_{Tr_{L'/L}}
  \ar[r]^{ } & {\zp(1)}\ar[d]_{=} \ar[r]^{ } & 0   \\
  0 \ar[r]^{ } & {\zp(1)}  \ar[r]^{ } & {\lu(L_\infty)}  \ar[r]^{\mathrm{Col}_{L,\epsilon}} & {\O_{L}\kl \Gamma \kr} \ar[r]^{ } & {\zp(1)} \ar[r]^{ } & 0   }
\ee

where $Tr_{L'/L}:\O_{L'}\kl \Gamma\kr\to \O_L\kl \Gamma \kr$ is induced by the trace on the
coefficients.    While the left and right square commute obviously, we sketch how to
check this   for the middle: 
%

Firstly note that by the uniqueness of the Coleman power series \[N_{L'/L}(g_{u'})=g_{N_{L'/L}(u')}\] for $u'\in\lu(L'_\infty),$ where
$N_{L'/L}:\O_{L'}\kl X\kr \to\O_{L}\kl X\kr$ is defined as $f(X)\mapsto \prod_{\sigma\in G(L'/L)}
f^\sigma(X),$ where $\sigma$ acts coefficient wise on $f$ (see the proof of \cite[Lem.\ 2]{Ya} for a similar argument). Secondly, one has
 \[\mathcal{L}(N_{L'/L}(g))=Tr_{L'/L}\mathcal{L}(g)\] for $g\in \O_{L'}\kl X\kr ^\times,$
since $N_{L'/L}$ and $\phi$ commute. Sofar we have seen that
\[Tr_{L'/L} \mathcal{L}(g_{u'})=\mathcal{L}(g_{N_{L'/L}(u')}),\]
which implies the claim
\[Tr_{L'/L}(\mathrm{Col}(u'))=\mathrm{Col}(g_{N_{L'/L}(u')})\]
using the defining equation \eqref{colpowerseries} and the compatibility of $Tr_{L'/L}$ with the Mahler transform  $\frak{M}:\O_K\kl\Gamma\kr\to \O_K\kl X\kr,\; \lambda\mapsto \lambda\cdot(1+X)$.

Taking inverse limits of \eqref{commCol}  we obtain the exact sequence \be\label{colemaninfinite}
\xymatrix@C=0.5cm{
  0 \ar[r] & {\lu(K_\infty)} \ar[rr]^(0.3){\mathcal{L}_{K ,\epsilon}} && {\Tun(K_\infty)\otimes_\Lambda \Lambda_{[\Tun,\tau_p]^{-1}}} \ar[rr]^{ } && {\zp(1)} \ar[r] & 0. }
\ee
Similarly, starting with \eqref{col2bis} we obtain the exact sequence
\be\label{colemaninfinitebis}
\xymatrix@C=0.5cm{
  0 \ar[r] & {\lu(K_\infty)} \ar[rr]^(0.3){-\mathcal{L}_{K ,-\epsilon}} && {\Tun(K_\infty)\otimes_\Lambda \Lambda_{[\Tun,\tau_p]^{-1}}} \ar[rr]^{ } && {\zp(1)} \ar[r] & 0. }
\ee

\subsection{Galois cohomology}

The complex $R\Gamma(\qp,\Tun(K_\infty))$ of continuous cochains has only   non-trivial cohomology
groups for $i=1,2$:
\be\label{localKummer}\H^1(\qp,\Tun(K_\infty))=\projlim{\qp\subseteq L\subseteq K_\infty
\hbox{\footnotesize finite}} \H^1(L,\zp(1))=\projlim{L} (L^\times)^{\wedge p}\ee by Kummer theory
and \be \label{localTate} \H^2(\qp,\Tun(K_\infty))=\projlim{\qp\subseteq L\subseteq K_\infty
\hbox{\footnotesize finite}} \H^2(L,\zp(1))=\zp\ee by local Tate-duality; here the sign of the
trace map $\mathrm{tr}:\H^2(\qp,\Tun(K_\infty))\cong \zp$ is normalised according to \cite[Ch. II,
\S 1.4]{kato-lnm} as follows: if $\theta \in \H^1(\qp,\Lambda)$ denotes the character $\xymatrix{
  G_{\qp} \ar[r]^{w} & \hat{\z}\ar[r]^{canon } & \Lambda   },$  where $w$ is the map which sends
  $Frob_p$ to $1$ and the inertia subgroup to $0$,  then we have a commutative diagram
  \be\label{normalisation}\xymatrix{
    {\mathbb{Q}_p^\times} \ar[d]_{\delta} \ar[r]^{v} & {\z}   \ar[r]^{canon } & {\zp}   \\
    {\H^1(\qp,\zp(1))} \ar[rr]^{-\cup\theta} &   & {\H^2(\qp,\zp(1))}\ar[u]^{\cong}_{\mathrm{tr}}   ,}\ee
    where $v$ denotes the normalised
  valuation map and $\delta$ is the Kummer map. Furthermore, the first isomorphism \eqref{localKummer} induces
\begin{description}
  \item[-] a canonical exact sequence \be\label{localunitsfinte} \xymatrix@C=0.5cm{
  0 \ar[r] & {\lu (K_\infty)} \ar[rr]^{ } && {\H^1(\qp,\Tun(K_\infty))} \ar[rr]^(0.7){-\hat{v}} && {\zp}\ar[r] & 0,
  }\ee
  if $K/\qp$ is finite, $\hat{v}$ being induced from the valuation maps $v_L:L^\times\to \z$ (the sign before $\hat{v}$ will become evident by the  descent calculation \eqref{loclocalunitsfinte}),
  \item[-] an isomorphism \be \label{localunitsinfinite}\lu(K_\infty)\cong
  \H^1(\qp,\Tun(K_\infty)),\ee if $p^\infty | [K:\qp].$
\end{description}

\subsection{Determinants}\label{secdet}

Now we assume that $K/\qp$ is {\em infinite.} Then
\[G\cong G'\times\Delta,\]
where $\Delta$ is a finite abelian group of order $d$ prime to $p$ and $G'\cong\zp^2.$ Thus
\[\Lambda(G)=\zp[\Delta]\kl\zp^2\kr\]
is a product of regular, hence  Cohen-Mac Caulay rings. Setting $\O:=\zp[\mu_d]$ we have
\[\Lambda(G)\subseteq \Lambda_\O(G)=\prod_{\chi \in \mathrm{Irr}_{\overline{\qp}}(\Delta)} \La_\O(G')
e_\chi,\] where $e_\chi$ denote the idempotents corresponding to $\chi,$ while
$\mathrm{Irr}_{\overline{\qp}}(\Delta)$ denotes the set of $\overline{\qp}$-rational characters of
$\Delta.$ Since regular rings are normal - or by Wedderburn theory -, it follows that there is a
product decomposition into local regular integral domains
\[\La(G)=\prod_{\chi \in
\mathrm{Irr}_{{\qp}(\Delta)}} \La_{\O_\chi}(G') e_\chi,\] where now $\mathrm{Irr}_{ {\qp}}(\Delta)$
denotes the set of $ {\qp}$-rational characters and $\O_\chi$ is the ring of integers of
$K_\chi:=\mathrm{End}_{\zp[\Delta]}(\chi).$

For the various rings $R$ showing up like $\Lambda(G)$ for different $G,$ we fix compatible
determinant functors $\d_R: \mathrm{D}^{\mathrm{p}}(R)\to \mathcal{P}_R$
 from the category of perfect complexes of $R$-modules (consisting of (bounded) complexes of finitely
 generated $R$-modules, quasi-isomorphic to strictly perfect complexes, i.e., bounded complexes of
 finitely generated projective $R$-modules) into the Picard category $\mathcal{P}_R$ with unit
 object $\u_R=\d_R(0),$  see Appendix \ref{determinants} 
for the yoga of determinants used in this article. 

\begin{lem}\label{cantriv}
For all $r\in\z$ there exists a canonical isomorphism \[\xymatrix{
  {\u_\Lambda \ar[r]^(0.35){\mathrm{can}_{\zp(r)}} }&    {\d_\Lambda(\zp(r))}.}\]
\end{lem}
\begin{rem}
The proof will show  that the same result holds for $G\cong \zp^k\times \Delta,$ $k\geq 2$ and any
$\La(G)$-module $M$ of Krull codimension at  least $2.$
\end{rem}

\begin{proof}
Since \[\Ext_{\Lambda(G)}^i(\zp(r),\Lambda(G))\cong\Ext^i_{\La(G')}(\zp(r),\La(G'))=0\]  for $i\neq
k (=2)$ we see that the codimension of $\zp(r)$  equals $k+1-1=k \geq 2.$ Setting $M=\zp(r)$ we
first show that the class $[M]$ in $G_0(\Lambda)=K_0(\Lambda)$ vanishes, i.e., there exists an
isomorphism $c_0:\u\cong \d(M)$ by the definition of $\mathcal{P}_R$ in \cite{fukaya-kato}. Since
\[K_0(\Lambda)=\bigoplus_{\chi}K_0(\Lambda_{\O_\chi}(G'))\cong \bigoplus_{\chi} \z,\] where the
last map is given by the rank, the claim follows, because $e_\chi M$ are torsion
$\La_{\O_\chi}(G')$-modules. By the knowledge of  the codimension we have $M_\mathfrak{p}=0$ for
all prime ideals $\p\subset \Lambda$ of height at most $1.$ In particular, we obtain canonical
isomorphisms
\[c_\p:\u_{\Lambda_\p}\cong \d_{\Lambda_\p}(M_\p).\] Since
$\mathrm{Mor}(\u_{\Lambda_\p},\d_{\Lambda_\p}(M_\p))$ is a (non-empty) $K_1(\Lambda_\p)$-torsor,
there exist unique $\lambda_\p\in \Lambda_\p^\times= K_1(\Lambda_\p)$ such that
\[c_\p=(c_0)_\p \cdot\lambda_\p,\] where $(c_0)_\p=\Lambda_\p\otimes_\La c_0.$ Now let
$\mathfrak{q}=\mathfrak{q}_\chi$ be a prime of height zero corresponding to $\chi\in
\mathrm{Irr}_{\qp}(\Delta).$ Then
\begin{eqnarray*}
c_{\mathfrak{q}}&=&\Lambda_{\mathfrak{q}}\otimes_{\Lambda_\p} c_\p \\
&=&\Lambda_{\mathfrak{q}}\otimes_{\Lambda} c_0  \cdot\lambda_\p= (c_0)_{\mathfrak{q}}\lambda_\p
\end{eqnarray*}
for all prime ideals $\p\supset \mathfrak{q}$ of height one, whence
\[\lambda_\p=\lambda_{\mathfrak{q}}.\] Thus \[\lambda_{\mathfrak{q}}\in \bigcap_{\p\supset \mathfrak{q}, \mathrm{ht}(\p)=1} \Lambda_\p^\times=\Lambda_{\O_\chi}(G')^\times\]
($\Lambda_{\O_\chi}(G')$ being 
regular, i.e., $\bigcap_{\p\supset \mathfrak{q}, \mathrm{ht}(\p)=1} \Lambda_\p
=\Lambda_{\O_\chi}(G') $) and \[\mathrm{can}_M:=(c_0\cdot \lambda_{\mathfrak{q}_\chi})_\chi:\u_\La
\to \d_\La(M)\] is unique and independent of the choice of $c_0.$ Here we used the canonical
decomposition $K_1(\La(G))\cong\bigoplus_\chi K_1(\La_{\O_\chi}(G')).$
\end{proof}

Now we shall finally define the $\epsilon$-isomorphism for the pair $(\La(G),\Tun):$ \be\label{epsTun}
\epsilon_{\La}(\Tun):=\epsilon_{\La,\epsilon}(\Tun):\u_\La \to
\d_\La(R\Gamma(\qp,\Tun))\d_\La(\Tun\otimes_\La \La_{\tau_p}):\ee Since $\La$ is regular we obtain
by  property B.h) in the Appendix
\begin{eqnarray*}
\d_\La(R\Gamma(\qp,\Tun))^{-1}&\cong& \d_\La(\H^1(\qp,\Tun))\d_\La(\H^2(\qp,\Tun))^{-1} \\
&\cong& \d_\La(\lu(K_\infty))\d_\La(\zp)^{-1}\\
&\cong& \d_\La(\Tun\otimes_\La \La_{\tau_p})\d_\La(\zp(1))^{-1}\d_\La(\zp)^{-1} \\
&\cong& \d_\La(\Tun\otimes_\La \La_{\tau_p}),
\end{eqnarray*}
where we used \eqref{localTate}, \eqref{localunitsinfinite} in the second equality,
\eqref{colemaninfinitebis},i.e., in particular the map $-\mathcal{L}_{K,\epsilon^{-1}}$ (sic!) and the regularity in the
third, while the identifications $\mathrm{can}_{\zp(1)}$ and $\mathrm{can}_{\zp }$ in the last
step. This induces \eqref{epsTun}.

In the spirit of \cite{fukaya-kato} this can be reformulated in a way that also covers
non-commutative rings $\La$ later. For any $a\in K_1(\widetilde{\La})$ define
 \[K_1(\La)_a:=\{x\in K_1(\widetilde{\La})| (1\otimes \phi)_*(x)=a \cdot x\},\]
which is non-empty by \cite[prop.\ 3.4.5]{fukaya-kato}. If $\La$ is the Iwasawa-algebra of an
abelian $p$-adic Lie-group, i.e., $K_1(\widetilde{\La})=\widetilde{\La}^\times,$ this implies in
particular that $\La_a\cap \widetilde{\La}^\times =K_1(\La)_a\neq \emptyset,$ whence we obtain an
isomorphism of $\widetilde{\La}$-modules \be \label{twisttilde} \La_a\otimes_\La\widetilde{\La}
\cong \widetilde{\La}, \;\; x\otimes y\mapsto x\cdot y.\ee Thus, one immediately sees, that the map
\[\lu(K_\infty)\to \Tun\otimes_\La \La_{\tau_p}\subseteq \Tun\otimes_\La \widetilde{\La}\]

extends to an exact sequence of $\widetilde{\La}$-modules \be
\label{colemaninfinitebasechange}\xymatrix@C=0.5cm{
  0 \ar[r] & {\lu(K_\infty)\otimes_\La\widetilde{\La}}  \ar[rr]^{ } && {\Tun\otimes_\La \widetilde{\La}} \ar[rr]^{ } && {\zpur(1)} \ar[r] & 0
  ,}\ee
  which in fact is canonically isomorphic to the base change of \eqref{colemaninfinite} from $\La$-
  to $\widetilde{\La}$-modules. Therefore   base changing     \eqref{epsTun} by $\widetilde{\La}\otimes_\La -$ and using
\eqref{twisttilde} (tensored with $\Tun(K_\infty)$) we obtain \be\label{epsTun'}
\epsilon'_\La(\Tun):=\epsilon'_{\La,\epsilon}(\Tun):\u_{\widetilde{\La}} \to
\d_\La(R\Gamma(\qp,\Tun))_{\widetilde{\La}}\d_\La(\Tun)_{\widetilde{\La}},\ee which actually arises
as base-change from some
\[\epsilon_0:\u_\La \to \d_\La(R\Gamma(\qp,\Tun(K_\infty))\d_\La(\Tun(K_\infty))\]
plus a twisting by an element $\delta\in K_1(\La)_{\tau_p},$ i.e.,
\[\epsilon'_\La(\Tun) \in \mathrm{Mor}(\u_\La, \d_\La(R\Gamma(\qp,\Tun(K_\infty))\d_\La(\Tun(K_\infty)))\times^{K_1(\La)} K_1(\La)_{\tau_p}.\]
Indeed, fixing an isomorphism $\psi: \La\cong  \La_{\tau_p}$ (cf.\ Proposition \ref{twist}) sending
$1$ to $\delta$, \eqref{twisttilde} implies that $\delta\in K_1(\La)_{\tau_p}$ and the claim
follows from the commutative diagram
\[\xymatrix{
  {\Tun\otimes_\La \widetilde{\La}} \ar[r]^{\Tun\otimes \delta^{-1}} & {\Tun\otimes_\La \widetilde{\La}}   \\
  {\Tun\otimes_\La \La_{\tau_p}} \ar[r]^{\Tun\otimes \psi^{-1}}\ar@{^(->}[u]_{ }  &   {\Tun\otimes_\La  {\La}}\ar@{^(->}[u]_{ }  }\]

( $ \epsilon'_\La(\Tun)$  equals $\delta $ times the base change of
$\epsilon_0:=(\Tun\otimes\psi^{-1})\circ \epsilon_\La(\Tun)$).

\subsection{Twisting}

We recall the following definition from \cite[\S1.4]{fukaya-kato}.

\begin{defn}
A ring $R$ is of
\begin{description}
  \item[(type 1)] if there exists a two sided ideal $I$ of $R$ such that $R/I^n$ is finite of order a power of
  $p$ for any $n \geq 1$ and such that $R\cong \projlim{n} R/I^n.$
  \item[(type 2)] if $R$ is the matrix-algebra $M_n(L)$ of some finite extension $L$ over $\qp$ and
  some dimension $n\geq 1.$
\end{description}
\end{defn}

By lemma 1.4.4 in (loc.\ cit.) $R$ is of type 1 if and only if the defining condition above holds
for the Jacobson ideal $J=J(R).$ Such rings are always semi-local and $R/J$ is a finite product of
matrix algebras over finite fields.

Now let $R$ be a commutative ring of type 1 and let $\T=\T_\chi$ be a free $R$-module of rank one
with Galois action given by
\[\chi=\chi_\T: G_{\qp} \to R^\times\]
which factors through $G.$ By $\tilde{\chi}_\T$ we denote the induced ring homomorphism $\La(G)\to
R.$ Furthermore let $Y=Y_\chi$ be the $(R,\La(G))$-bimodule which is $R$ as $R$-module and where
$\La(G)$ is acting via \[\chi_Y:=\tilde{\chi}_\T^{-1}\chi_{cyc}:\La(G)\to R\] (from the right)
where
\[\chi_{cyc}: \La(G) \to \zp\to R\] is induced by the cyclotomic character and the unique ring
homomorphism $\zp\to R.$  Then the map
\[\xymatrix{
  Y\otimes_{\La(G)} \Tun \ar[r]^(0.7){\cong} &\T, }\;\;  y\otimes t\mapsto y\cdot \chi_Y(t),\] is   an
isomorphism of $R$-modules which is Galois equivariant, where the Galois action on the tensor
product is given by $\sigma(y\otimes t)=y\otimes \sigma(t)$ for $\sigma \in G_{\qp}.$

Let $\widetilde{R}$ and $R_a$ be defined in the same way as for $\La.$ Then, using the isomorphisms
\[Y\otimes_\La \d_\La(R\Gamma(\qp,\Tun ))\cong \d_R(R\Gamma(\qp,Y\otimes_\La\Tun ))\cong
\d_R(R\Gamma(\qp,\T ))\] by \cite[1.6.5]{fukaya-kato} and
\[R\otimes_\La \La_a\cong R_{\chi (a)},\] where $\chi:\La\to R$ denotes a continuous ring homomorphism, we may define the following
$\epsilon$-isomorphisms.

\begin{defn}
In the above situation we set \[\epsilon_{R}(\T):=\epsilon_{R,\epsilon}(\T):=Y\otimes_\Lambda
\epsilon_{\Lambda,\epsilon} (\Tun):\u_R \to \d_R(R\Gamma(\qp,\T ))\d_R(\T \otimes_R
R_{\chi(\tau_p)})\] and
\[\epsilon'_R(\T):=\epsilon'_{R,\epsilon}(\T):=Y\otimes_\Lambda \epsilon'_{\Lambda,\epsilon}(\Tun): \u_{\widetilde{R}} \to\d_R(R\Gamma(\qp,\T ))_{\widetilde{R}}\d_R(\T )_{\widetilde{R}}.\]
\end{defn}

By definition we have the following important twist invariance \be\label{twistinvariance}
Y'\otimes_R \epsilon_R(\T)=\epsilon_{R'}(\T')\mbox{ and }Y'\otimes_R
\epsilon'_R(\T)=\epsilon'_{R'}(\T')\ee for any $(R',R)$-bimodule $Y'$ which is projective as
$R'$-module and satisfies $Y'\otimes_R \T \cong \T'.$ Here $R$ and $R'$ denote commutative rings of
type $1$ or $2$. Indeed, to this end the definition extends to  all pairs $(R,\T)$ where $R$ is a (not
necessarily commutative) ring of type 1 or 2 and $\T$ stands for a projective $R$-module such that
there exists a $(R,\Lambda)$-bimodule $Y$ which is projective as $R$-module and such that $\T\cong
Y\otimes_\Lambda \Tun.$ In this context we denote by $[\T,\sigma],$ $\sigma\in G_{\qp}$, the
element in $K_1(R)$ induced by the action of $G_{\qp}$ on $\T$; note that this induces a
homomorphism $[\T,-]:G(\mathbb{Q}_p^{ab})\to K_1(R).$

\begin{ex}
Let $\psi: G_F\to \mathbb{Z}_p^\times$ be a Gr\"{o}ssencharacter of an imaginary quadratic field $F$
 such that $p$ is split in $F$ and assume that its restriction to $G_{F_\nu},$  $\nu$ a place above $p,$ factors through $G.$ We write
$\T_\psi$ for the free rank one $\La(G)$-module with Galois action given by
$\sigma(\lambda)=\lambda\bar{\sigma}^{-1}\psi(\sigma).$   Then we also write
$\epsilon_\La(\psi)$ for $\epsilon_\La(\T_\psi).$
\end{ex}

\subsection{The $\epsilon$-conjecture}

We fix $K/\mathbb{Q}_p$ infinite and recall that $G=G(K_\infty/\mathbb{Q}_p)$ as well as
$\La=\La(G)$  and $\La_\O=\La_\O(G)$ for $\O=\O_L$ the ring of integers of some finite extension
$L$ of $\mathbb{Q}_p.$ If $\chi:G\to \O_L^\times$ denotes any continuous character such that the representation
\[ V_\chi:=L(\chi),\]
   whose underlying  vector space is just $L$ and whose
  $G_{\qp}$-action is given by $\chi$, is de Rham, hence potentially semistable by \cite{serre} (in this classical case) or by \cite{berger02} (in general), then
we have
\[L\otimes_{\O_L} \epsilon'_{\O_L}(\T_\chi)=\epsilon'_{L }(V_\chi)\]
by definition. The $\epsilon$-isomorphism conjecture by Fukaya and Kato in \cite[conjecture
3.4.3.]{fukaya-kato} states that
\begin{equation}
\label{conjecture}  \epsilon'_{L }(V_\chi) =\Gamma_L(V_\chi)\cdot \epsilon_{L,\epsilon,dR}(V_\chi)\cdot
\theta_L(V_\chi),
\end{equation}
where, for any de Rham $p$-adic representation $V$ of $G_{\qp}$,

a) $\Gamma_L(V):=\prod_\z \Gamma^*(j)^{-h(-j)}$  with $h(j)=\dim_L gr^j D_{dR}(V)$ and
\[\Gamma^*(j)=\left\{
                \begin{array}{ll}
                  (-1)^j(-j)!^{-1}, & \hbox{$j\leq 0;$} \\
                  \Gamma(j), & \hbox{$j>0$,}
                \end{array}
              \right.
\]
denotes the leading coefficient of the $\Gamma$-function,

b) the map
\[\epsilon_{dR}(V):=\epsilon_{L,\epsilon,dR}(V):\u_{\widetilde{L}}\to\d_{\widetilde{L}}(V)\d_{\widetilde{L}}(D_{dR}(V))^{-1},\]
with $\widetilde{L}:= \qnr\otimes_{\qp} L$ is defined in   \cite[prop.\
3.3.5]{fukaya-kato}. 
We shall recall it  after  the proof of Lemma \ref{epsilondR},

c) $\theta_L(V)$ is defined as follows: Firstly, $\r_f(\qp,V)$ is defined as a certain subcomplex
of  the local cohomology complex $\r(\qp,V),$ concentrated in degree $0$ and $1,$ whose image in
the derived category is isomorphic to
\begin{gather} \label{rf-l} \r_f(\qp,V)\cong [\xymatrix{ D_{cris}(V) \ar[r]^<(0.1){(1-\varphi_p,1)}
& D_{cris}(V)\oplus D_{dR}(V)/D_{dR}^0(V)  }]     \end{gather} Here $\varphi_p$ denotes the usual
Frobenius homomorphism
and the induced map $t(V):=
D_{dR}(V)/D_{dR}^0(V)\to H^1_f(\qp,V)$ is the exponential map $exp_{BK}(V)$ of Bloch-Kato, where we
write $\H^n_f(\qp,V)$ for the cohomology of $\r_f(\qp,V).$ Now
\be\label{Theta}\theta_L(V):\u_L\to\d_L(\r(\qp,V))\cdot \d_L(D_{dR}(V))\ee is by definition induced
from $\eta_p(V)\cdot \overline{(\eta_p(V^*(1))^*)}$ (see Remark \ref{inverse}  for the notation) - with
\begin{eqnarray}\label{eta}
\eta_p(V):& \u_{L} &\to
 \d_{L}(\mathrm{R\Gamma}_f(\Q_p,V))\d_L(t(V)),
\end{eqnarray}
arising by trivializing  $D_{cris}(V)$   in \eqref{rf-l} by the identity - followed by an
isomorphism induced by   local Tate-duality   \be \label{local-finite-dual}\r_f(\ql,V)\cong
\big(\r(\ql,V^*(1))/\r_f(\ql,V^*(1))\big)^*[-2] \ee

 and  using    $D_{dR}^0(V)=t(V^*(1))^*.$

   More explicitly, $\theta_p(V)$ is obtained from applying
the determinant functor to the following exact sequence \[\begin{split} &\xymatrix@C=0.5cm{
0\ar[r]^{ } & {\H^0(\qp,V)} \ar[r]^{ } & D_{cris}(V) \ar[r]^{ } &
D_{cris}(V)\oplus t(V) \ar[rr]^<(0.3){exp_{BK}(V)} && {\H^1(\qp,V)} \ar[r]^{ } &     }\\
&\xymatrix@C=0.5cm{ \ar[rr]^<(0.1){exp_{BK}(V^*(1))^* } && D_{cris}(V^*(1))^*\oplus
t(V^*(1))^*\ar[r]^{ } & D_{cris}(V^*(1))^* \ar[r]^{ } & {\H^2(\qp,V)} \ar[r]^{ } & 0  }
\end{split}\] which arises from joining the defining sequences of $exp_{BK}(V)$ with the dual
sequence for $exp_{BK}(V^*(1))$ by local duality \eqref{local-finite-dual}.

%

\begin{rem} a) The $\epsilon$-conjecture may analogously be formulated using $\epsilon_R(\T)$ instead
$\epsilon_R'(\T).$ In the following we will amply switch between the two versions.

b) Since by definition of $\epsilon_{\O_L}(\T_\chi)$ we have
\begin{eqnarray*}
L\otimes_{\O_L}\epsilon'_{\O_L}(\T_\chi)&=& L\otimes_{\O_L}\left(Y_\chi\otimes_\Lambda
\epsilon'_{\La}(\T_{un})\right)\\
&=& (L\otimes_{\O_L}Y_\chi)\otimes_\La \epsilon'_{\La}(\T_{un})
\end{eqnarray*}
\eqref{conjecture} amounts to showing that \be\label{vchi}  L\otimes_{\La}\epsilon_\La
(\T_{un})=\epsilon_L(V_\chi)\ee holds, where $\Lambda$ acts on $L$ via
$\chi^{-1}\chi_{cyc}:\La(G)\to \O_L\subseteq L.$ Once we have shown \eqref{vchi} for all possible
$\chi$ as above, it follows immediately by twisting that e.g.\  $\epsilon_\La(\T_{K_\infty}(T))$
for $T=\zp(\eta)(r)$ as below satisfies the descent property \[V_\rho\otimes_\La \epsilon_\La(\T_{K_\infty}(T))=\epsilon_L(V (\rho^*))\] with $V (\rho^*):=V \otimes_{\qp} V_{\rho^*}$ for all
one-dimensional representation $V_\rho$ arising from some continuous $\rho:G\to \O_L^\times$ and
its contragredient representation $V_{\rho^*}.$
\end{rem}

Note that by \cite{serre} any $V_\chi$ as above is of the form
\[W=L(\eta \rho)(r)=Lt_{\rho\eta,r},\]
where $r$ is some integer, $\eta:G\to \O_L^\times $   is an unramified character and
$\rho:G\twoheadrightarrow G(K'_m/\qp)\to \O_L^\times$ denotes an Artin-character for some finite
  subextension $K'$ of $K/\qp$ and with $m=a(\rho)$ chosen minimal, i.e.,
$p^{a(\rho)}$ is the $p$-part of the conductor of $\rho.$

In the following we fix $\eta$ and $r$ and we set $T:=\zp(\eta)(r),$ $V:=T\otimes_{\zp}\qp$ and
\begin{equation}
\label{defomation}
\T_{K_\infty}=\T_{K_\infty}(T)= \Lambda^\sharp\otimes_{\zp} T,\end{equation} the free $\Lambda$-module on which
$\sigma\in G_{\qp}$ acts as $\bar{\sigma}^{-1}\eta\kappa^r(\sigma).$


Now we are going to make the map \eqref{Theta} explicit. First we describe the local cohomology
groups: \be\label{H0} \H^0(\qp,W)=\left\{
              \begin{array}{ll}
                L, & \hbox{if $r=0, \rho\eta=\1;$} \\
                0, & \hbox{otherwise.}
              \end{array}
            \right.
\ee By local Tate duality we have
 \be\label{H2} \H^2(\qp,W)\cong\H^0(\qp,W^*(1))^*=\left\{
                                                   \begin{array}{ll}
                                                     L, & \hbox{$r=1,\rho\eta=\1;$} \\
                                                     0, & \hbox{otherwise.}
                                                   \end{array}
                                                 \right.
\ee From the local Euler-Poincar\'{e}-characteristic formula one immediately obtains \be\label{H1}
\dim_L\H^1(\qp, W)= \dim_L\H^1(\qp, W^*(1))=\left\{
                                                           \begin{array}{ll}
                                                             2, & \hbox{$r=0$ or $1, \rho\eta=\1 $ ;} \\
                                                             1, & \hbox{otherwise.}
                                                           \end{array}
                                                         \right.
\ee Following the same reasoning as in \cite[lem.~ 1.3.1.]{Benois-NQD} one sees that
\begin{align*}
\label{H1f} &\H^1_f(\qp,W)\cong \left(\H^1(\qp,W^*(1))/\H^1_f(\qp,W^*(1))\right)^*=\\
&\qquad\left\{
                                                                                          \begin{array}{ll}
                                                                                            \H^1(\qp,W), & \hbox{$r\geq 2,$ or $r=1$ and $\rho\eta\neq \1;$} \\
                                                                                            \mathrm{im}\left(\lu(\qp) \otimes_{\zp} \qp \to \H^1(\qp,\qp(1))\right), & \hbox{$r=1, \rho\eta=\1;$} \\
                                                                                            \H^1(\mathbb{F}_p,\qp) & \hbox{$r=0,\rho\eta=\1;$} \\
                                                                                            0, & \hbox{$r\leq -1,$ or $r=0$ and $\rho\eta\neq\1.$}
                                                                                          \end{array}
                                                                                        \right.
\end{align*}
where the map in the second line is the Kummer map. Hence we call the cases  $r=0 \mbox{ or } 1,
\rho\eta=\1$   {\em exceptional} and all the others {\em generic}.

For the tangent space we have by \eqref{deRhamfilt} \be\label{t} t(W)=\left\{
                    \begin{array}{ll}
                      D_{dR}(W)=L, & \hbox{$r>0;$} \\
                      0, & \hbox{$r\leq 0.$}
                    \end{array}
                  \right.
\ee and
 \be\label{t*} t(W^*(1))=\left\{
                    \begin{array}{ll}
                      0, & \hbox{$r>0;$} \\
                      D_{dR}(W^*(1))=L, & \hbox{$r\leq 0.$}
                    \end{array}
                  \right.
\ee while
 \be\label{crys} D_{cris}(W)=\left\{
                              \begin{array}{ll}
                                0, & \hbox{$a(\rho)\neq 0;$} \\
                                Le_{\rho\eta,r}, & \hbox{otherwise,}
                              \end{array}
                            \right.
\ee with Frobenius action given as $\phi(e_{\rho\eta,r})=p^{-r}\rho\eta(\tau_p^{-1})e_{\rho\eta,r}.$

\subsubsection{The  case $r\geq 1$}

In this case we have  $\Gamma_L(W)=\Gamma(r)^{-1}=(r-1)!^{-1}$ and $\H^0(\qp,W)=0,$ whence
\be\label{Dcrys} 1-\phi: D_{cris}(W)\to D_{cris}(W)\ee and \be\label{exp} \exp(W):D_{dR}(W)\cong
H^1_f(\qp,W)\ee are   bijections. Thus, combined with the exact sequences
\be\label{dualexp}\xymatrix@C=0.4cm{
  0 \ar[r] &   \H^1_f(\qp, W^*(1))^* \ar[rrr]^{exp(W^*(1))^*} &&& D_{cris}(W^*(1))^* \ar[r]^{1-\phi^*} & D_{cris}(W^*(1))^* \ar[r]^{ } & \H^2(\qp,W) \ar[r] & 0
  }\ee
  and
\[\xymatrix@C=0.5cm{
  0 \ar[r] & \H^1_f(\qp,W) \ar[rr]^{ } && \H^1(\qp,W) \ar[rr]^{ } && \H^1_f(\qp,W^*(1))^* \ar[r] & 0 }\]
they induce the following isomorphism corresponding to $\theta_L(W)^{-1}:$
\begin{align*}
\d_L(D_{dR}(W)) \to\d_L(R\Gamma(\qp,W))^{-1}.
\end{align*}
In the {\it generic} case it decomposes as
\[\d_L(exp(W)): \d_L(D_{dR}(W)) \to\d_L(\H^1(\qp,W))=\d_L(R\Gamma(\qp,W))^{-1}\]
times
\[\frac{\det(1-\phi^*|D_{cris}(W^*(1))^*)}{\det(1-\phi|D_{cris}(W))}:\u_L \to \u_L,\]   which
equals
 \be \frac{\det(1-\phi|D_{cris}(W^*(1)))}{\det(1-\phi|D_{cris}(W))}=
\left\{
\begin{array}{ll}
  \frac{1-p^{r-1}\rho\eta(\tau_p)}{1-p^{-r}\rho\eta(\tau_p^{-1})}, & \hbox{if $a(\rho)=0;$} \\
   1, & \hbox{otherwise.}
  \end{array}
  \right.
 \ee

Now let $r=1$ and $\rho\eta=\1,$ i.e., we consider the {\em exceptional} case $W=\qp(1)$. As now
$\det(1-\phi|D_{cris}(W^*(1)))=0$ and the two occurrences of $D_{cris}(W^*(1))^*$ in
 \eqref{dualexp} are identified via the identity, the map $\theta_L(W)^{-1}$ is also induced by
 \eqref{Dcrys},\eqref{exp} together with the (second) exact sequence in the following commutative
 diagram
 \be\label{theta1}\xymatrix{
   0   \ar[r]^{ } & {\mathbb{Z}_p^\times \otimes\qp} \ar[d]_{\cong }^{\delta} \ar[r]^{ } & {\widehat{\mathbb{Q}_p^\times} \otimes\qp} \ar[d]^{{\delta} }_{\cong} \ar[r]^{{\color{black} -}\hat{v}\otimes\qp} &
   {\qp}
   \ar[r]^{ } & 0\phantom{,} \\
   0 \ar[r]^{ } & {\H^1_f(\qp,\qp(1))} \ar[r]^{ } & {\H^1(\qp,\qp(1))} \ar@{.>}[r]^{ } & {\H^2(\qp,\qp(1))}\ar[u]_{Tr}^{\cong} \ar[r]^{ } & 0,   }
\ee where the first two vertical maps $\delta$ are induced by Kummer theory, $v$ denotes the
normalised valuation map and the dotted arrow is defined by commutativity: I.e., $\theta_L(W)^{-1}$ arises from
\be\label{theta2}\xymatrix{ { \d_{\qp}( D_{dR}(\qp(1)))}
 \ar[r]^(0.3){exp_{\qp(1)}} & \d_{\qp}(H^1_f(\qp,\qp(1)))\cong  \d_{\qp}(R\Gamma(\qp,W))^{-1}  }
 \ee
 times
\be \label{theta3}\det(1-\phi|D_{cris}(\qp(1))) =(1-p^{-1}) . \ee
Combining \eqref{theta1},  \eqref{theta2} and \eqref{theta3} this can rephrased as follows:

\begin{prop}
The map $\theta(\qp(1))$ is just induced by the single exact sequence
\be\label{thetaall}\xymatrix{
       0 \ar[r]^{ } & {t( \qp(1))\cong\qp} \ar[rr]^{(1-p^{-1})^{-1}\exp_{\qp(1)} } && {\H^1(\qp,\qp(1))} \ar[r]^{{\color{black} -}\hat{v}\otimes\qp } & {\H^2(\qp,\qp(1))}  \ar[r]^{ } &
       0.
   }\ee
\end{prop}

\begin{proof}
Since $t(\qp)=0,$ it follows directly from its definition as connecting homomorphism that
\[exp_{\qp}:\qp=D_{cris}(\qp)\to \H^1_f(\qp,\qp)\subseteq \H^1(\qp,\qp)\] sends $\alpha\in\qp$ to
the character $\chi_\alpha:G_{\qp}\to\qp,$ $ g\mapsto (g-1)c,$ where $c\in \widehat{\qp^{nr}}$
satisfies $(1-\varphi)c=\alpha,$ i.e., $\chi_\alpha(\phi)=-\alpha.$  As noted in \cite[lem.\
1.3.1]{Benois-NQD}, we thus may identify $\H^1_f(\qp,\qp)=\H^1(\fp,\qp).$ Identifying the copies of
$D_{cris}(\qp)$ (in the dual of \eqref{dualexp}) gives rise to a map
\[\psi:\qp=\H^0(\qp,\qp)\to\H^1_f(\qp,\qp),\; \alpha \mapsto \chi_\alpha.\] By local Tate duality
\[\xymatrix{
  {\H^1(\qp,\qp(1))/\H^1_f(\qp,\qp(1))} \ar[d]_{\psi^*} \phantom{mmm}\times  & {\H^1_f(\qp,\qp)}  \ar[r]^{ } & {\H^2(\qp,\qp(1))\cong \qp} \ar@{=}[d]^{ } \\
  {\phantom{/\H^1_f(\qp,\q}} {\H^2(\qp,\qp(1))}  \phantom{mmmp(1))}\times & {\H^0(\qp,\qp) }\ar[r]^{ }\ar[u]_{\psi} & {\H^2(\qp,\qp(1))\cong \qp}   }\]
we obtain for the dual map $\psi^*$ using the normalisation \eqref{normalisation}
\[\mathrm{tr}(\psi^*(\delta(p))=\mathrm{tr}(\delta(p)\cup \chi_1)=\chi_1(\phi)=-1.\] Thus the
dotted arrow in \eqref{theta1} being $\psi^*$ this diagram commutes as claimed.
\end{proof}

\subsubsection{The  case $r\leq 0$}

This case is just dual to the previous one replacing $W$ by $W^*(1).$

\subsection{The descent}\label{descent}

Let $K$ be infinite. In order to describe the descent of $\L_{K,\epsilon^{-1}}$   in
\eqref{colemaninfinite} we set
\be\label{LT} \L_\T:=\L_{\T,\epsilon^{-1}}:=Y\otimes_\La\L_{K,\epsilon^{-1}},\ee
if the
projective left $\Lambda'$-module $Y$ (with commuting right $\La$-module structure) satisfies
$Y\otimes_\La\T_{un}\cong \T$ as $\La'$-module. As $\L_{K,\epsilon^{-1}}$ is the crucial
ingredient in the definition of $\epsilon'_\La(\T),$ the following descent diagram will be
important:

For fixed $\rho$ as before we choose $K'\subseteq K$ and $n\geq \max\{1, a(\rho)\}$ such that
$\rho$ factorises over $G_n:=G(K'_n/\qp).$ Setting $\La':=\qp[G_n]$ and
$V':=\qp[G_n]^\sharp\otimes\qp(\eta)(r)$   we first note that
\[\H^i(\qp,V')\cong \H^i(K_n',\qp(\eta)(r))\] by Shapiro's Lemma.
 Also, let $Y'$ be the
$(\La',\La)$-bimodule such that $Y'\otimes_\La\T_{un}\cong V'.$ We write $e_\chi:=\frac{1}{\#G_n}\sum_{g\in G_n} \chi(g^{-1})g$ for the usual idempotent which
induces a canonical decomposition $\Lambda'\cong\prod L_\chi$ into a product of finite extensions
$L_\chi$ of $\qp.$ In particular, for $L=L_\rho$ we have $W\cong
e_{\rho^{-1}}V'=L_\rho(\rho\eta)(r).$

Then, for $r\geq 1  $ and with $\Gamma(V'):=\bigoplus_\chi \Gamma(e_\chi V'),$ we have a
commutative diagram
\[\xymatrix{
   {Y'\otimes_\La \H^1(\qp,\T_{un})} \ar[d]_{pr_n} \ar[rr]^{-\L_{V'}=-Y'\otimes_\La\L_{K,\epsilon^{-1}}} &     &  Y'\otimes_\La \T_{un}\otimes_\La \La_{[\T_{un}, \tau_p^{-1}]} \ar[d]_{\cong}^{pr_n}   \\
  {\H^1(K_n',\qp(\eta)(r)) } & D_{dR}(V') \ar[l]_(0.4){\Gamma(V')^{-1}exp_V'} & V'\otimes_{\La'} (\La')_{[V',\tau_p^{-1}]}  \ar[l]_{ }   }\]

of $\La'$-modules as will be explained in the Appendix, Proposition \ref{commutative}.

Applying the exact functor $V_{\rho^*}\otimes_{\La'}-$ leads to the final commutative descent
diagram - at least for $W\neq \qp(1)$ -
 \be\label{descentdiag}\xymatrix{
   {Y''\otimes_\La \H^1(\qp,\T_{un})} \ar[d]_{pr_n} \ar[rrrr]^{-\L_{W}=-Y''\otimes_\La\L_{K,\epsilon^{-1}}} &     &&&  Y''\otimes_\La \T_{un}\otimes_\La \La_{[\T_{un},\tau_p^{-1}]} \ar[d]_{\cong}^{pr_n}   \\
  {\H^1(\qp,W) } & D_{dR}(W) \ar[l]_{\Gamma(W)^{-1}exp_W} &&& W\otimes_{L} L_{[W,\tau_p^{-1}]}  \ar[lll]_{\frac{\det(1-\varphi|D_{cris}(W^*(1)))}{\det(1-\varphi|D_{cris}(W ))}\cdot\epsilon_{L,\epsilon,dR}(W)^{-1}}
  }\ee
where $Y'':=V_{\rho^*}\otimes_{\La'}Y'=V_{\rho^*}\otimes_{\La}Y$ is a $(L,\La)$-bimodule. For
$W=\qp(1)$ the Euler factor in the nominator as well as the map $pr_0$ become zero, therefore we
shall instead apply a direct descent calculation in Lemma \ref{lemLmodpi} using semisimplicity and a
Bockstein homomorphism.

For the descent we need two ingredients:

\begin{itemize}
  \item the long Tor-exact sequence by applying $Y''\otimes_{\La(G)}-$ to the defining sequence
  \eqref{colemaninfinite}  for $- \mathcal{L}_{K,\epsilon^{-1}}$ and
  \item the convergent cohomological spectral sequence
  \be\label{spectralsequence} E_2^{i,j}:=\mathrm{Tor}_{-i}^\Lambda(Y'',H^j(\qp,\T_{un}))\Rightarrow \H^{i+j}(\qp,W), \ee
  which is induced from the isomorphism
  \[Y''\otimes^\mathbb{L}R\Gamma(\qp,\T_{un})\cong R\Gamma(\qp,Y''\otimes_\Lambda\T_{un}),\] proved in \cite{fukaya-kato} and using $W\cong Y''\otimes\T_{un}$,
\end{itemize}
together with the fact \cite{ven-det} that the determinant functor is compatible with both. Since
for $\T=\T(T):= \Lambda^\sharp\otimes_{\zp}T\cong Y\otimes_\Lambda \Tun$ \be
\H^i(\qp,\T)\cong\left\{
                   \begin{array}{ll}
                     T, & \hbox{if $i=0$ and $r=0, \eta=\1$;} \\
                     \H^1(\qp,\T)\neq 0, & \hbox{if $i=1$;} \\
                     T(-1), & \hbox{if $i=2$ and $r=1, \eta=\1$;} \\
                     0, & \hbox{otherwise.}
                   \end{array}
                 \right.
\ee we obtain for $r\geq 1$ the following exact sequence of terms in lower degree
\begin{multline}\label{lowdegree}
    \xymatrix{
      0 \ar[r]^{ } & {\mathrm{Tor}_1^\Lambda(Y'',\H^1(\qp,\T_{un}))} \ar[r]^{ } & {\H^0(\qp,W)} \ar[r]^{ } & {\mathrm{Tor}_2^\Lambda( Y'',\H^2(\qp,\T_{un}))}
      }\\
      \xymatrix{
         \ar[r]& Y''\otimes_\Lambda\H^1(\qp,\T_{un}) \ar[r]^{ } & {\H^1(\qp,W)} \ar[r]^{ } &  {\mathrm{Tor}_1^\Lambda(Y'',\H^2(\qp,\T_{un}))} \ar[r]^{ } & 0   }
\end{multline}
and
\[{\mathrm{Tor}_2^\Lambda(Y'',\H^1(\qp,\T_{un}))}=0 \mbox{ and } Y''\otimes_\Lambda\H^2(\qp,\T_{un})\cong \H^2(\qp,W)\]
 or, as $Y''\otimes^\mathbb{L}_\Lambda R\Gamma(\qp,\T_{un})\cong
V_{\rho^*}\otimes^\mathbb{L}_\Lambda\left(Y\otimes_\Lambda R\Gamma(\qp, \T_{un})\right)\cong
V_{\rho^*}\otimes^\mathbb{L}_\Lambda R\Gamma(\qp, \T),$ which is canonically isomorphic to
\begin{multline}\label{lowdegree2}
    \xymatrix{
      0 \ar[r]^{ } & {\mathrm{Tor}_1^\Lambda(V_{\rho^*},\H^1(\qp,\T))} \ar[r]^{ } & {\H^0(\qp,W)} \ar[r]^{ } & {\mathrm{Tor}_2^\Lambda( V_{\rho^*},\H^2(\qp,\T))}
      }\\
      \xymatrix{
         \ar[r]& V_{\rho^*}\otimes_\Lambda\H^1(\qp,\T) \ar[r]^{ } & {\H^1(\qp,W)} \ar[r]^{ } &  {\mathrm{Tor}_1^\Lambda(V_{\rho^*},\H^2(\qp,\T))} \ar[r]^{ } & 0   }
\end{multline}
and \[{\mathrm{Tor}_2^\Lambda(V_{\rho^*},\H^1(\qp,\T ))}=0 \mbox{ and }
V_{\rho^*}\otimes_\Lambda\H^2(\qp,\T )\cong \H^2(\qp,W).\]

In the {\em generic case} the spectral sequence boils down to the following isomorphism
\begin{align}\label{ss-generic}
Y''\otimes_\Lambda\H^1(\qp,\T_{un})\cong \H^1(\qp,W).
\end{align}

Considering the support of $\zp(1)$ one easily sees that $\mathrm{Tor}_i^\Lambda(Y'',\zp(1))=0$ for
all $i\geq 0.$ Hence the long exact Tor-sequence associated with \eqref{colemaninfinite} combined
with \eqref{localunitsinfinite} degenerates to \be\label{Tor-generic}\xymatrix{
  Y''\otimes_\Lambda \H^1(\qp,\Tun)\ar[r]^{-\L_W}_{\cong} & W\otimes_L L_{[W,\tau_p]^{-1} }  }
\ee while for all $i\geq 0$ \be
\mathrm{Tor}_i^\Lambda(Y'',\H^2(\qp,\Tun))=\mathrm{Tor}_i^\Lambda(Y'',\zp)=0. \ee

Thus the conjectured equation \eqref{conjecture} holds by \eqref{descentdiag}, \eqref{ss-generic},
 \eqref{Tor-generic} and the definition \eqref{epsTun} with $-\mathcal{L}_{K,\epsilon^{-1}}$.

For the {\em exceptional case} $W=\zp(1)$  we set $R=\Lambda(\Gamma)_\frak{p},$ where $\frak{p}$
denotes the augmentation ideal of $\Lambda(\Gamma)$ and recall that $R$ is a discrete valuation
ring with uniformising element $\pi:=1-\gamma_0,$ where $\gamma_0$ is a fixed element in $\Gamma,$
which is sent to $1$ under $\xymatrix{
  {\Gamma} \ar[r]^{\kappa} & {\mathbb{Z}_p^\times} \ar[r]^{\log_p} & {\zp }  },$ and residue field
  $R/\pi=\qp.$ The commutative diagram of homomorphisms of rings
  \[\xymatrix{
    {\Lambda=\Lambda(G) }\ar[d]_{ } \ar[r]^{ } & {\Lambda(\Gamma)} \ar[d]_{ } \ar[r]^{ } & R \ar[d]^{ } \\
    {\zp} \ar[r] &  {\zp}\ar[r]^{ } & {\qp}   }\]
induces with $Y'':=R/\pi$
\begin{eqnarray}
R\Gamma(\qp,\qp(1))&\cong& Y''\otimes_\Lambda^{\mathbb{L}} R\Gamma(\qp,\Tun(K_\infty))\label{bockstein}\\
&\cong& {\qp} \otimes_{\Lambda(\Gamma)}^{\mathbb{L}}\left(\Lambda(\Gamma)\otimes_\Lambda^\mathbb{L} R\Gamma(\qp,\Tun(K_\infty))\right)\notag\\
&\cong& {\qp} \otimes_{R}^{\mathbb{L}} R\otimes_{\Lambda(\Gamma)}\otimes_\Lambda^\mathbb{L} R\Gamma(\qp,\Tun(\mathbb{Q}_{p,\infty})) \notag\\
&\cong& {\qp} \otimes^\mathbb{L}_R R\Gamma(\qp,\Tun(\mathbb{Q}_{p,\infty}))_\frak{p}.\notag
\end{eqnarray}

In particular, the descent calculation factorises over the cyclotomic level, i.e.,
\[\epsilon'_{\qp}(\qp(1))=R/\pi\otimes_R\epsilon'_{R}(R\otimes_{\Lambda(\Gamma)}\Tun(\mathbb{Q}_{p,\infty}))\]
is induced by $\epsilon'_{R}(R\otimes_{\Lambda(\Gamma)}\Tun(\mathbb{Q}_{p,\infty}))$ which in turn
is induced by the localisation at $\frak{p}$ of the exact sequences (all for $K=\qp$)
$\eqref{col2}$ \be\xymatrix{
  {\lu(\mathbb{Q}_{p,\infty})_\frak{p}} \ar[rrr]_(0.4){\cong}^(0.4){-\mathcal{L}_{\Tun(\mathbb{Q}_{p,\infty})_\frak{p}}} &&& {\Tun(\mathbb{Q}_{p,\infty})_\frak{p}\otimes_R R_{[\Tun,\tau_p]^{-1}}}   }
\ee (this arises as long exact Tor-sequence from \eqref{colemaninfinite}) and
\eqref{localunitsfinte} \be \label{loclocalunitsfinte}\xymatrix@C=0.5cm{
  0 \ar[r] & {\lu (\mathbb{Q}_{p,\infty})_\frak{p}} \ar[rr]^{ } && {\H^1(\qp,\Tun(\mathbb{Q}_{p,\infty}))_\frak{p}} \ar[rr]^(0.45){{\color{black} -}\hat{v}} && {\qp\cong \H^2(\qp,\Tun(\mathbb{Q}_{p,\infty}))_\frak{p}}\ar[r] &
  0.
  }\ee
The latter one  arises from an analogue of the spectral sequence \eqref{spectralsequence} above -
which gives with $\mathcal{H}=G(K_\infty/\mathbb{Q}_{p,\infty})$ an exact sequence
\[\xymatrix@C=0.5cm{
  0 \ar[r] & {\H^1(\qp,\Tun(K_\infty))_\mathcal{H}} \ar[r]^{ } & {\H^1(\qp,\Tun(\mathbb{Q}_{p,\infty}))} \ar[r]^(0.7){} & {\H^2(\qp,\Tun(K_\infty))^\mathcal{H} }\ar[r] & 0,
  }\]
  and
  \[ \H^2(\qp,\Tun(K_\infty))_\mathcal{H}=\H^2(\qp,\Tun(\mathbb{Q}_{p,\infty})),\]
 - combined with \eqref{localTate}    and an identification of $\H^2(\qp,\Tun(K_\infty))^\mathcal{H}=\zp$ with
 \linebreak
 $\H^2(\qp,\Tun(K_\infty))_\mathcal{H}=\zp$ induced by the base change of $\mathrm{can}_{\zp}$. Indeed, it is easy to check that the long exact $\mathcal{H}$-homology
($=\mathrm{Tor}^\Lambda_i(\Lambda(\Gamma),-)$) sequence associated with \eqref{colemaninfinite}
recovers \eqref{col2}, in particular $\H^1(\qp,\Tun(K_\infty))_\mathcal{H}\cong
\lu(K_{\infty})_H\cong \lu(\mathbb{Q}_{p,\infty}).$ Moreover, the composite
\[\tilde{\beta}:
 {\H^1(\qp,\Tun(\mathbb{Q}_{p,\infty}))} \to {\H^2(\qp,\Tun(K_\infty))^\mathcal{H} } = {\H^2(\qp,\Tun(K_\infty))_\mathcal{H}} = {\H^2(\qp,\Tun(\mathbb{Q}_{p,\infty}))   }\]
is via   restriction and taking $G(K_\infty^{\mathcal{H}'}/\mathbb{Q}_{p,\infty})$-invariants by
construction induced by  the Bockstein homomorphism $\beta$ associated to the exact triangle in the
derived category
\[\xymatrix{
  R\Gamma(\qp,\Tun(K_\infty)) \ar[r]^{1-h_0} & R\Gamma(\qp,\Tun(K_\infty))\ar[r]^{ } & R\Gamma(\qp,\Tun(K_\infty^{\mathcal{H}'})),   }\]
  where $\mathcal{H}'$ is the maximal pro-$p$ quotient of $\mathcal{H}$ and $h_0$ is the image of
  $\phi$. By \cite[lem.\ 5.9]{flach-survey} (and the argument following directly afterwards using  the projection
  formula for the cup product) it follows that $\tilde{\beta}$ is given by the cup product $\theta\cup -$ where \[\theta:G_{\mathbb{Q}_{p,\infty}} \twoheadrightarrow \mathcal{H}'\cong\zp\]
is the unique character such that   $h_0$ is sent to $1$ under the second isomorphism. Using our
above convention of the trace map \eqref{normalisation} one finds according to \cite[ch.\ II, \S
1.4.2]{kato-lnm} that the above composite equals $-\hat{v}$. Indeed
\[\mathrm{tr}(\tilde{\beta}(\delta(p)))=\mathrm{tr} (\theta\cup \delta(p))=-\theta(\phi)=-\theta(h_0)=-1.\]

 Now consider the element
\[u:=(1-\epsilon_n^{-1})_n\in \projlim n  {(\qp(\mu_{p^n})^\times)^\wedge}\cong\H^1(\qp,\Tun(\mathbb{Q}_{p,\infty}))\]
and its image $u_\frak{p}$ in  $\H^1(\qp,\Tun(K_\infty))_\frak{p}.$
\begin{lem}\label{lemLmodpi}
$\H^1(\qp,\Tun(K_\infty))_\frak{p}\cong R u_\frak{p}$ is a free $R$-module of rank one and
$\mathcal{L}_{\Tun(\mathbb{Q}_{p,\infty})_\frak{p}}$ induces modulo $\pi$ an canonical isomorphism
\be\label{Lmodpi}\xymatrix{
  t(\qp(1))  & {\lu(\mathbb{Q}_{p,\infty})_\frak{p}/\pi} \ar[l]_{-\mathcal{L}_{\qp(1)}}\ar[r]^{} &  {\qp} \ar[r]^{ } & {\H^1(\qp,\Tun(K_\infty))_\frak{p}/\pi }
  }\ee which sends $(1-p^{-1})e\in \qp e=t(\qp(1))$ to $\bar{u},$ the image of $u_\frak{p}$ (but
  which is of course not induced by the map ${\lu (\mathbb{Q}_{p,\infty})_\frak{p}}  \to {\H^1(\qp,\Tun(K_\infty))_\frak{p}}
  $ as the latter  map becomes trivial modulo $\pi$ !).
\end{lem}

\begin{proof}
First note that $\bar{u}$ is mapped under the natural inclusion
$\H^1(\qp,\Tun(\mathbb{Q}_{p,\infty}))_\frak{p}/\pi\subseteq  \H^1(\qp,\qp(1))$ to the image of $p$
under the isomorphism $(\mathbb{Q}_p^\times)^\wedge\otimes \qp\cong  \H^1(\qp,\qp(1))$ of Kummer
theory, because $p$ is the image of the elements $1-\epsilon^{-1}_n$ under the norm maps. In
particular, $\bar{u}$ is nonzero. By \eqref{localunitsfinte} the element $u^{\gamma_0-1}$ belongs
to $\lu(\mathbb{Q}_{p,\infty}).$ In order to calculate the image of the class
$\overline{u^{\gamma_0-1}}$ of $u^{\gamma_0-1}_\frak{p}$ modulo $\pi$ under $-\mathcal{L}_{\qp(1)}$
we note that \[g(X)=g_{u^{\gamma_0-1},-\epsilon}(X)=\frac{(1+X)^{\kappa(\gamma_0)}-1}{X}\equiv
\kappa(\gamma_0)  \mbox{ mod } (X),\] whence we obtain from setting $X=0$ in
$-\mbox{\eqref{colpowerseries}}$ that
\[-(1-p^{-1})=-(1-p^{-1})\log(\kappa(\gamma_0))=-\mathrm{Col}_{-\epsilon}(u^{\gamma_0-1})\cdot 1\] equals the image of
$\overline{u^{\gamma_0-1}}$ in $\qp=R/\pi \otimes_{\Lambda(\Gamma)} \Tun(\mathbb{Q}_{p,\infty})
\otimes_{\Lambda(\Gamma)}  {\Lambda(\Gamma)}_{[\Tun,\tau_p]^{-1}}$ under $-\mathcal{L}_{\qp(1)}$. In particular,
$\overline{u^{\gamma_0-1}}$ is a basis of ${\lu(\mathbb{Q}_{p,\infty})_\frak{p}/\pi},$ which is
mapped to zero in ${\H^1(\qp,\Tun(\mathbb{Q}_{p,\infty}))_\frak{p}/\pi },$ whence  the long exact
Tor-sequence associated with \eqref{loclocalunitsfinte} induces the isomorphisms
\[\xymatrix{
  {\H^1(\qp,\Tun(\mathbb{Q}_{p,\infty}))_\frak{p}/\pi }\ar[r]^(0.7){{\color{black}-}v} & {\qp}   },\; \bar{u}\mapsto {\color{black}-}1, \mbox{ as $v(p)=1$,}\]
  and
  \[\xymatrix{
    {\qp} \ar[r]^{ } & {\lu(\mathbb{Q}_{p,\infty})_\frak{p}/\pi}  },\; 1\mapsto \overline{u^{\gamma_0-1}},
    \]  where the latter formula follows from the snake lemma.
By the first isomorphism and Nakayama's lemma the first statement is proven and therefore
\[\H^1(\qp,\Tun(\mathbb{Q}_{p,\infty}))_\frak{p}[\pi]=\lu(\mathbb{Q}_{p,\infty})_\frak{p}[\pi]=0.\]
The second claim follows now from the composition of these isomorphisms.
\end{proof}

Finally, the  exact triangle in the derived category of $R$-modules
\[\xymatrix@C=0.5cm{
  R\Gamma(\qp,\Tun(\mathbb{Q}_{p,\infty}))_\frak{p} \ar[r]^{1-\gamma_0} & R\Gamma(\qp,\Tun(\mathbb{Q}_{p,\infty}))_\frak{p}\ar[r]^{} & {\zp}\otimes^\mathbb{L}_{\Lambda(\Gamma)} R\Gamma(\qp,\Tun(\mathbb{Q}_{p,\infty}))_\frak{p} \ar[r]^{} &  }\]
combined with \eqref{bockstein} induces the Bockstein map $\beta=\theta\cup$ sitting in the
canonical exact sequence (depending on $\gamma_0$) \be\label{bocksteinexplizit}\xymatrix{
  0   \ar[r]^{ } & {\H^1(\qp,\Tun(K_\infty))_\frak{p}/\pi }   \ar[r]^{ } & {\H^1(\qp,\qp(1))} \ar[d]_{\cong} \ar[r]^{\beta} & {\H^2(\qp,\qp(1))} \ar[r]_{ }
  \ar[d]^{\cong} & {0, }   \\
    &   & (\mathbb{Q}_p^\times)^\wedge\otimes \qp \ar[r]^{\log_p} & {\qp}   &     }
\ee where $\theta$ denotes the composite $\xymatrix@C=0.5cm{
  G_{\qp} \ar[r]^{\kappa} & {\mathbb{Z}_p^\times} \ar[r]^{\log_p} & {\zp}   }$ considered as element in  $\H^1(\qp,\zp),$
  see \cite[lem.\ 5.7-9]{flach-survey}, \cite[\S 3.1]{burns-ven1} and \cite[\S 5.3]{bf4} and for the commutativity of the square \cite[Ch.\ II,1.4.5]{kato-lnm}. The right zero in the upper line comes from $\H^3(\qp,\Tun(K_\infty))_\frak{p}[\pi]=0$.
Combining with \eqref{Lmodpi} it follows that $\epsilon'_{\qp}(\qp(1))$ is induced from the exact
sequence
\be\label{epsilonqp1} \xymatrix{
  0   \ar[r]^{ } & t(\qp(1))\cong\qp   \ar[rr]^{(1-p^{-1})^{-1}\mathrm{im}(p) } && {\H^1(\qp,\qp(1))}  \ar[r]^{\beta=\log_p} & {\H^2(\qp,\qp(1))} \ar[r]_{ }
   &  0, }   \ee
which does {\em not} coincide at all with \eqref{thetaall} (not even up to sign), nevertheless they
induce the same map on determinants: both induce a map
\[{\d_{\qp}(\H^1(\qp,\qp(1)))}\to \d_{\qp}(\H^2(\qp,\qp(1)))\otimes\d_{\qp}(t(\qp(1))\cong \d_{\qp}(\qp)\otimes\d_{\qp}(\qp)\]
sending $(1-p^{-1})^{-1}\exp(1)\wedge
{\color{black}-}\mathrm{im}(p)=(1-p^{-1})^{-1}\mathrm{im}(p)\wedge \exp(1)$ to $  1\wedge 1.$ This
completes the proof in the exceptional case.

For $r\leq 0$ one has symmetric calculations - at least in the generic case - using a descent
diagram analogous to \eqref{descentdiag} except   the left map on the bottom being now induced by
the dual Bloch-Kato exponential map $ \Gamma(V) exp_{V^*(1)}^*$ as indicated in
\eqref{descentberger2} (left to the reader). The exceptional case can be dealt with  by using the
following duality principle (generalised reciprocity law) as follows.

Let $\T$ be a free $R$-module of rank one with compatible $G_{\qp}$-action as above. Then
\[\T^*:=\Hom_{R}(\T,R)\] is a free $R^\circ$-module of rank one - for the action
$h\mapsto h(-)r,$ $r$ in the opposite ring $  R^\circ$ of $R$ - with compatible $G_{\qp}$-action
given by $h\mapsto h\circ \sigma^{-1}$. Recall that in Iwasawa theory we have the canonical
involution $\iota:\Lambda^\circ\to \Lambda$, induced by $g\mapsto g^{-1},$ which allows to consider
(left) $\Lambda^\circ$-modules again as (left) $\Lambda$-modules, e.g.\ one has $\T^*(T)^\iota\cong
\T(T^*)$ as $(\Lambda, G_{\qp})$-module, where $M^\iota:=\Lambda\otimes_{\iota,\Lambda^\circ}M$
denotes the $\Lambda$-module with underlying abelian group $M,$ but on which $g\in G$ acts as
$g^{-1}$ for any $\Lambda^\circ$-module $M.$

Given $\epsilon'_{R^\circ,-\epsilon}(\T^*(1))$ we may  apply the dualising functor $-^*$ (compare \
B.j) in   Appendix \ref{determinants}) to obtain an isomorphism
\[\epsilon'_{R^\circ,-\epsilon}(\T^*(1))^*:(\d_{R^\circ}(R\Gamma(\qp,\T^*(1) ))_{\widetilde{R^\circ}})^*(\d_{R^\circ}(\T^*(1) )_{\widetilde{{R^\circ}}})^* \to\u_{\widetilde{{R^\circ}}},\]
while the local Tate duality isomorphism \cite[\S 1.6.12]{fukaya-kato}
\[\psi(\T): R\Gamma(\qp,\T)\cong R\Hom_{R^\circ}(R\Gamma(\qp,\T^*(1)),R^\circ)[-2]\]   induces an
isomorphism
\begin{align}
\overline{\d_R(\psi(\T))_{\widetilde{R }}}^{-1}:&\left((\d_{R^\circ}(R\Gamma(\qp,\T^*(1)
))_{\widetilde{R^\circ}})^*\right)^{-1}\cong\\
&\d_R(R\Hom_{R^\circ}(R\Gamma(\qp,\T^*(1)),R^\circ))_{\widetilde{R }}^{-1} \to
\d_R(R\Gamma(\qp,\T))_{\widetilde{R }}^{-1}\notag\end{align} where we use the notation of Remark
\ref{inverse}. Consider the product
\[\epsilon'_{R,\epsilon}(\T)\cdot \epsilon'_{R^\circ,-\epsilon}(\T^*(1))^*\cdot \overline{\d_R(\psi(\T))_{\widetilde{R }}}^{-1}:\d_R(\T(-1))_{\widetilde{R }}\cong\d_R(\T^*(1)^*)_{\widetilde{R }}\to\d_R(\T)_{\widetilde{R }}\]
and the isomorphism $\xymatrix{   {\T(-1)}   \ar[r]^{\cdot\epsilon} & \T   }$   which sends
$t\otimes
  \epsilon^{\otimes -1}$ to $t.$

\begin{prop}[Duality]\label{duality}
Let $\T$ be as above such that $\T\cong Y\otimes_\Lambda\Tun$ for some $(R,\Lambda)$-bimodule $Y,$
which is projective as $R$-module. Then
\[\epsilon'_{R,\epsilon}(\T)\cdot \epsilon'_{R^\circ,-\epsilon}(\T^*(1))^*\cdot \overline{\d_R(\psi(\T))_{\widetilde{R }}}^{-1}=\d_R\left(\xymatrix{
  {\T(-1) }  \ar[r]^(0.6){\cdot\epsilon} & \T   }\right)_{\widetilde{R}}.\]
\end{prop}

\begin{proof}
First  note that the statement is stable under applying $Y'\otimes_R-$ for some $(R',R)$-bimodule
$Y'$, which is projective as a $R'$-module by the functoriality of local Tate duality and the lemma
below. Thus we are reduced to the case $(R,\T)=(\Lambda,\T(T))$ where $T=\zp(r)(\eta)$ is generic.
Since the morphisms between $\d_R(\T(-1))_{\widetilde{R }}$ and $\d_R(\T)_{\widetilde{R }}$ form a
$K_1(\widetilde{\Lambda})$-torsor and the kernel
\[SK_1(\widetilde{\Lambda}):=\ker\left(K_1(\widetilde{\Lambda})\to \prod_{\rho\in{\rm{Irr}}({G})} K_1(\widetilde{L_\rho})\right)=1\]
is trivial, as $G$ is abelian, it suffices to check the statement for all $(L, V(\rho)),$ which is
nothing else than the content of \cite[prop.\ 3.3.8]{fukaya-kato}. Here $ {\rm{Irr}}({G})\ $
denotes the set of $ \overline{\mathbb{Q}}_p$-valued irreducible representations of $ {G}$ with
finite image.
\end{proof}

\begin{lem}
Let  $Y$ be a $(R',R)$-bimodule such that $Y\otimes_R\T\cong \T'$ as $(R',G_{\qp})$-module and let
$Y^*=\Hom_{R'}(Y,R')$ the induced $(R'^\circ,R^\circ)$-bimodule. Then there is a natural
\begin{enumerate}
\item  equivalence of functors \[Y\otimes_R\Hom_{R^\circ}(-, R^\circ)\cong \Hom_{R'^\circ}(Y^*\otimes_{R^\circ}
-,R'^\circ)\] on $P(R^\circ).$
\item isomorphism $Y^*\otimes_{R^\circ} \T^*\cong (\T')^*$ of $(R'^\circ,G_{\qp})$-modules.
\end{enumerate}
\end{lem}

\begin{proof}
This is easily checked using the adjointness of $\Hom$ and $\otimes.$
\end{proof}

\begin{prop}[Change of $\epsilon$]\label{dependence}
Let $c\in \mathbb{Z}_p^\times$ and let $\sigma_c$ be the unique element of the inertia subgroup of
$G(\mathbb{Q}_p^{ab}/\qp)$ such that $\sigma_c(\epsilon)=c\epsilon$ (in the $\zp$-module $\zp(1),$
whence written additively). Then
\[\epsilon'_{R,c\epsilon}(\T)=[\T,\sigma_c]\epsilon'_{R,\epsilon}(\T)\]
\end{prop}

\begin{proof}
As in the proof of Proposition \ref{duality} this is easily reduced to the pairs $(L, V(\rho)),$
for which the statement follows from the functorial properties of $\epsilon$-constants \cite[\S
3.2.2(2)]{fukaya-kato}.
\end{proof}

Altogether we have proven the following

\begin{thm}[Kato, $\epsilon$-isomorphisms]\label{main}
Let $\T$ be   such that $\T\cong Y\otimes_\Lambda\Tun$ as $(R, G_{\qp})$-module for some
$(R,\Lambda)$-bimodule $Y$ which is projective as $R$-module, where $\La=\La(G)$ with $G=G(L/\qp)$
for any $L\subseteq \mathbb{Q}_p^{ab}$ . Then the epsilon isomorphism $\epsilon'_R(\T)$ exists
satisfying the twist-invariance \eqref{twistinvariance}, the descent property \eqref{conjecture},
the ``change of $\epsilon$'' relation \ref{dependence} and the duality relation in \ref{duality}. In
particular $\epsilon'_\Lambda(\T)$ exists for all
pairs $(\Lambda,\T),$ with $\T\cong\Lambda$ one-dimensional (free) as $\Lambda$-module.
\end{thm}

\begin{proof}
For $G$ a two-dimensional $p$-adic Lie group this has been shown explicitly above. The general case
follows by taking limits.
\end{proof}

We shortly indicate how this result implies the validity of a local main conjecture in this context. Here again we restrict to the universal case $\Tun$, but we want to point out that similar statements hold for general $\T $ as in the above theorem by the twisting principle, in particular it applies to $\T_E$ for the local representation given by a CM elliptic curve as in Example \ref{exCM} below.

In the situation of section \ref{secdet} - in particular $G$ is a two-dimensional $p$-adic Lie group -  let \[S:=\{\lambda\in\Lambda \mid \Lambda/\Lambda\lambda \mbox{ is finitely generated over } \Lambda(G(K_\infty/\mathbb{Q}_{p,\infty}))\} \]  denote the canonical Ore set of $\Lambda$ (see \cite{cfksv}) and similarly $\widetilde{S}$ the canonical Ore set of $\widetilde{\Lambda}.$ Fix an element $u$ of $\lu(K_\infty)=\H^1(\qp,\Tun(K_\infty))$ such that the map $\Lambda\to \H^1(\qp,\Tun(K_\infty)), 1\mapsto u,$ becomes an isomorphism after base change to $\widetilde{\Lambda}_{\widetilde{S}}$ (such ``generators'' exist according to \eqref{colemaninfinitebasechange} and Proposition \ref{twist}). Then, setting  $\mathbb{L}=-\mathcal{L}_{K,\epsilon^{-1}}$,
\[
\epsilon'_\La(\Tun) :\u_{\widetilde{\La}} \to
\d_\La(R\Gamma(\qp,\Tun))_{\widetilde{\La}}\d_\La(\Tun)_{\widetilde{\La}},\] induces a map
\be
 \u_{\widetilde{\La}} \to
\d_\La(\H^1(\qp,\Tun)/\Lambda u)_{\widetilde{\La}}^{-1}\d_\La(\H^2(\qp,\Tun) )_{\widetilde{\La}}\d_\La(\Tun/\mathbb{L}(u))_{\widetilde{\La}},\ee
such that its base change followed by the canonical trivialisations  (all arguments on the right hand side are $\tilde{S}$-torsion modules!)
\begin{align*}
 \u_{\widetilde{\La}_{\widetilde{S}}} \to
\d_\La(\H^1(\qp,\Tun)/\Lambda u)_{\widetilde{\La}_{\widetilde{S}}}^{-1}\d_\La(\H^2(\qp,\Tun) )_{\widetilde{\La}_{\widetilde{S}}}&\d_\La(\Tun/\Lambda\mathbb{L}(u))_{\widetilde{\La}_{\widetilde{S}}}\cong\\
&\d_{\widetilde{\La}_{\widetilde{S}}}(\zpnr)\d_{\widetilde{\La}_{\widetilde{S}}}(\zpnr(1))\to\u_{\widetilde{\La}_{\widetilde{S}}}
\end{align*}
equals the identity in  $\mathrm{Aut}(\u_{\widetilde{\La}_{\widetilde{S}}} )=K_1(\widetilde{\La}_{\widetilde{S}})$ by Lemma \ref{cantriv}. Let $\mathcal{E}_u$ be the element in $K_1(\widetilde{\La}_{\widetilde{S}}) $ such that
\[\mathbb{L}(u)=\mathcal{E}_u^{-1}\cdot(1\otimes\epsilon).\] Consider the connecting homomorphism $\partial$ in the exact localisation sequence
\[\xymatrix{
  K_1(\widetilde{\La}) \ar[r]^{ } & K_1(\widetilde{\La}_{\widetilde{S}}) \ar[r]^{\partial} & K_0(\widetilde{S}\mbox{-tor}) \ar[r]^{ } & 0,   }\]
where $\widetilde{S}\mbox{-tor}$ denotes the category of finitely generated $\widetilde{\La}$-modules which are $\widetilde{S}$-torsion. Then we obviously have
\[\partial(\mathcal{E}_u)=-[\Tun/\La\mathbb{L}(u)]=[\H^2(\qp,\Tun)]-[\H^1(\qp,\Tun)/\Lambda u]\]
in $K_0(\widetilde{S}\mbox{-tor}).$  Moreover one can evaluate $\mathcal{E}_u$ at Artin characters $\rho$ of $G$ as in \cite{cfksv} and derive an interpolation property for $\mathcal{E}(\rho)$ from Theorem \ref{main} by the techniques of \cite[lem.\ 4.3.10]{fukaya-kato}; this will be carried out in \cite{schmitt}. These two properties build the local main conjecture as suggested by Fukaya and Kato in a much more general, not necessarily commutative setting.  Kato  (unpublished)  has shown that $\widetilde{\La}_{\widetilde{S}}\otimes_\La\lu(K_\infty) \cong \widetilde{\La}_{\widetilde{S}}$ does hold in vast generality for $p$-adic Lie extensions.

\section{The semi-local case}

Let $F_\infty/\Q$ be a $p$-adic Lie extension with Galois group $G$ and $\nu$ be any place of
$F_\infty$ above $p$ such that $G_\nu=G(F_{\infty,\nu}/\qp)$ is the decomposition group at $\nu.$
For any free $\zp$-module $T$ of finite rank with continuous Galois action by $G_\Q$ we define the
free $\La(G)$-module
\[\T:=\T(T)_{F_\infty}:=\La(G)^\natural\otimes_{\zp} T\]
with the usual diagonal $G_\Q$-action. Similarly, we define the free $\La(G_\nu)$-module
\[\T^{loc}:=\La(G_\nu)^\natural \otimes_{\zp} T\]
with the usual diagonal $G_{\qp}$-action. Then we have the following canonical isomorphism of $(\La(G),
G_{\qp})$-bimodules
\[\T\cong \La(G)\otimes_{\La(G_{\nu})} \T^{{loc}}. \]

Thus we might define
\[\epsilon_{\La(G)}(\qp,\T):=\La(G)\otimes_{\Lambda(G_\nu)} \epsilon'_{\Lambda(G_\nu)}(\T^{loc}): \u_{\widetilde{\La(G)}} \to\d_{\La(G)}(R\Gamma(\qp,\T ))_{\widetilde{{\La(G)}}}\d_{\La(G)}(\T )_{\widetilde{{\La(G)}}}.\]

Now let $\rho:G\to GL_n(\O_L)$ be a continuous map and $\rho_\nu$ its restriction to $G_\nu,$ where
$L$ is a finite extension of $\qp.$ By abuse of notation we shall also denote the induced ring
homomorphisms $\La(G)\to M_n(\O_L)$ and $\La(G_\nu)\to M_n(\O_L)$ by the same letters. Since we
have a canonical isomorphism
\[L^n\otimes_{\rho,\La(G)}\T\cong L^n \otimes_{\rho_\nu, \La(G_\nu)}\T^{loc}\] of
$(L,G_{\qp})$-bimodules, we obtain
\begin{align*}
 L^n\otimes_{\rho,\La(G)}\epsilon_{\La(G)}&(\qp,\T)= \\
 & L^n \otimes_{\rho_\nu, \La(G_\nu)} \epsilon'_{\Lambda(G_\nu)}(\T^{loc}): \u_{\widetilde{L}} \to\d_{L}(R\Gamma(\qp,V(\rho^*) ))_{\widetilde{{L}}}\d_{L}(V(\rho^*) )_{\widetilde{{L}}},
\end{align*}
where $V=T\otimes_{\zp}\qp.$

\begin{ex}\label{exCM}
Let $E$ be a  elliptic curve defined over $\Q$ with CM by the ring of integers of an imaginary
quadratic extension $K\subseteq F_\infty$ of $\Q$ and let $\psi$ denote the Gr\"{o}ssencharacter
associated to $E.$ Then $T_E\cong \Ind^K_\Q T_\psi,$ which is isomorphic to $T_\psi\oplus
T_{\psi^c}$ as representation of $G_K.$ Here $T_\psi$ equals $\zp$ on which $G_\Q$ acts via $\psi$,
while $\psi^c$ is the conjugate of $\psi$ by complex multiplication  $c\in G(K/\Q).$

Assuming $K_{\nu}=\qp$ and setting   $\T_E:=\T,\; \T_E^{loc}:= \T^{loc}$ for $T=T_E$ as well
as $\T_\psi:=\La(G)^\natural\otimes_{\zp} T_\psi,\; \T_\psi^{loc}:=\La(G_\nu)^\natural\otimes_{\zp}
T_\psi$ we obtain
\[\T_E \cong\T_\psi\oplus \T_{\psi^c}\] as $(\La(G),G_K)$-modules and hence
\begin{eqnarray*}
\epsilon_{\La(G)}(\qp,\T_E)&=&\epsilon_{\La(G)}(\qp,\T_\psi)\epsilon_{\La(G)}(\qp,\T_{\psi^c})\\
 &=& \Lambda(G)\otimes_{\La(G_\nu)}\left(\epsilon_{\La(G_\nu)}(\qp,\T_\psi^{loc})\epsilon_{\La(G_\nu)}(\qp,\T_{\psi^c}^{loc})\right).
\end{eqnarray*}
\end{ex}

If $F$ is a number field and $F_\infty$ a $p$-adic Lie-extension of $F$ again with Galois group
$G,$ then, for a place $\frak{p}$ above $p$ and a projective $\La(G)$-module $\T$ with continuous
$G_{F_\frak{p}}$-action we define a corresponding $\epsilon$-isomorphism
\[\epsilon_{\La(G)}(F_\frak{p},\T):\u_{\widetilde{\La(G)}} \to\d_{\La(G)}(R\Gamma(F_\frak{p},\T ))_{\widetilde{{\La(G)}}}\d_{\La(G)}(\T )^{[F_\frak{p}:\qp]}_{\widetilde{{\La(G)}}}
\] to be induced from
\begin{align*}
 \epsilon_{\La(G)}(\qp,&\z[G_{\qp}]\otimes_{\z[G_{F_\frak{p}}]}\T):\\
 & \u_{\widetilde{\La(G)}} \to\d_{\La(G)}(R\Gamma(\qp,\z[G_{\qp}]\otimes_{\z[G_{F_\frak{p}}]}\T ))_{\widetilde{{\La(G)}}}\d_{\La(G)}(\z[G_{\qp}]\otimes_{\z[G_{F_\frak{p}}]}\T
)_{\widetilde{{\La(G)}}}.
\end{align*}

Finally we put
\[\epsilon_\La (F\otimes_{\Q} \qp, \T)=\epsilon_\La(\qp,\bigoplus_{\frak{p}|p}\z[G_{\qp}]\otimes_{\z[G_{F_\frak{p}}]}\T)=\prod_{\frak{p}|p}
\epsilon_\La(F_\frak{p},\T),\]  where $\frak{p}$ runs through the places of $F$ above $p.$

\section{Global functional equation}

In this section we would like to explain the applications addressed   in the introduction. In the same setting as in Example \ref{exCM} we assume that $p$ is a prime of good ordinary reduction for the CM elliptic curve $E$ and we set $F_\infty=\q(E(p)),$  as well as $G=G(F_\infty/\q)$  and $\Lambda:=\Lambda(G)$. We write $M=h^1(E)(1)$ for the motive attached to $E$ and set
$\epsilon_{p,\La}(M)=\epsilon_{\La}(\qp,\T_E).$ Using \cite{yasuda} one obtains similarly $\epsilon$-isomorphisms over $\mathbb{Q}_l$, $l\neq p$ which we call analogously
$\epsilon_{l,\La}(M)$. Finally, one can define $\epsilon_{\infty,\La}(M)$ also at the place at infinity, see \cite[conj.\ \S 3.5]{fukaya-kato} or  - with a slightly different normalisation - at the end of \cite[\S 5]{ven-BSD}, we choose the latter one. Let $S$ be the finite set of places of $\Q$ consisting of $p,\infty$ as well as the places of bad reduction of $M.$

Now, according to the Conjectures of \cite{fukaya-kato} there exists a $\zeta$-isomorphism
\[\zeta_\La(M):=\zeta_\La(\T_E): \u_\La \to \d_\La(\r_c(U,\T_E))^{-1}\] which is the global analogue of the $\epsilon$-isomorphism concerning special $L$-values (at motivic points in the sense of \cite{flach-survey2}) instead of $\epsilon$- and $\Gamma$-factors; here $\r_c(U,\T_E)$ denotes the perfect complex calculating \'{e}tale cohomology with compact support of $\T_E$ with respect to $U=\mathrm{Spec}(\z)\setminus S$. Good evidence for the existence of $\zeta_\La(M)$  is given in (loc.\ cit.) although Flach concentrates on the commutative case, i.e., he considers $\Lambda(G(F_\infty/K))$ instead of $\Lambda(G)$; from this the non-commutative version probably follows by similar techniques as in \cite{bou-ven}, but as a detailed discussion would lead us too far away from the topic of this article, we just assume the existence here for simplicity. Then we obtain the following

\begin{thm}\label{functionalequation}   There is  the functional equation
\[\zeta_\La(M)=(\overline{\zeta_\La(M)^*})^{-1} \cdot \prod_{v\in S} \epsilon_{v,\La}(M).\]
\end{thm}

This result is motivated by \cite[conj.\ 3.5.5]{fukaya-kato}, for more details see \cite[thm.\ 5.11]{ven-BSD}, compare also with \cite[section 5]{bf}. Observe that we used the self-duality $M=M^*(1)$ of $M$ here.

Finally  we want to address the application towards the descent result with Burns mentioned in the introduction.  If $\omega$  denotes the Neron differential of $E,$ we obtain the
usual real and complex periods
  $\Omega_{\pm} = \mathop{\int}\limits_{\gamma^{\pm}} \omega$ by integrating along pathes $\gamma^{\pm}$ which
  generate $H_1(E(\Co),\Z)^{\pm}.$ We set   $R = \{q \,\, {\mbox{prime}}, | j(E) |_q > 1 \}\cup \{p\}$
  and let $u,w$ be the roots of the characteristic polynomial of the action of Frobenius on the Tate-module
  $T_E$ of $E$
 \[1- a_pT + pT^2 = (1 - uT) (1-wT),\;\; u \in \mathbb{Z}_{p}^{\times}.\]
  Further let   $p^{{\frak{f}}_p (\rho)}$ be the  $ p$-part of the conductor of an Artin representation $\rho,$ while
    $P_p (\rho, T) = \det (1 - Frob_p^{-1} \,\, T
  | V_{\rho}^{I_p})$   describes the Euler-factor of $\rho$ at $p.$ We also
  set   $d^{\pm} (\rho) =  \dim_\mathbb{C} V_{\rho}^{\pm}$ and denote by $\rho^*$ the contragredient representation of $\rho.$
By $e_p(\rho)  $ we denote the local $\epsilon$-factor of $\rho$ at
  $p. $ In the notation of \cite{Tate} this is
  $e_p(\rho,\psi(-x),dx_1)$ where $\psi$ is the additive character
  of $\Q_p$ defined by $x \rightarrow exp(2\pi i x)$ and $dx_1$ is
  the Haar measure that gives volume 1 to $\zp$. Moreover, we write
$R_\infty(\rho^*)$ and  $R_p(\rho^*)$ for the complex and
$p$-adic regulators of $E$ twisted by $\rho^*.$
  Finally, in order to express special values of complex $L$-functions in the $p$-adic world,
  we fix embeddings   of  $ \bar{\mathbb{Q}} $ into $ \mathbb{C}$ and $\mathbb{C}_p,$ the completion of an algebraic closure of $\Qp. $

In \cite[thm.\ 2.14]{bou-ven} we have shown that as a consequence of the work of Rubin and Yager there exists $\mathcal{L}_E\in K_1(\La_{\zp}(G)_S)$ satisfying the interpolation property
\[
\mathcal{L}_E (\rho) = \frac{L_R(E, \rho^*,
1)}{\Omega_+^{d^+ (\rho)} \Omega_-^{d^-(\rho)}}\, e_p( {\rho} )
\frac{P_p(\rho , u^{-1})} {P_p(\rho^*, w^{-1})}
u^{-{\frak{f}}_p ({\rho} )}
\]
for all Artin representations $\rho$ of $G.$ Moreover the (slightly non-commutative) Iwasawa Main Conjecture (see \cite{cfksv} or Conjecture 1.4 in loc.\ cit.) is true supposed the $\mathfrak{M}_{H}(G)$-Conjecture (see \cite{cfksv} or Conjecture 1.2 in loc.\ cit.) holds; for CM elliptic curves this conjecture is actually equivalent to the vanishing of the cyclotomic $\mu$-invariant of $E$.  In  \cite[Conj.\ 7.4/9, Prop.\ 7.8]{burns-ven2} a refined Main Conjecture has been formulated requiring the following $p$-adic BSD-type formula:

 At each Artin representation $\rho$ of $G $ (with coefficients in $L$) the leading term $\L^*_E(\rho)$ of $\L_E$ (as defined in \cite{burns-ven1}) equals
 \begin{equation}\label{padicBSD}(-1)^{{r(E)(\rho^*)}}\frac{L_{R}^*(E,\rho^*)R_p(\rho^*)}{\Omega_+^{d_+(\rho)}\Omega_-^{d_-(\rho)}R_\infty(\rho^*)}\, e_p( {\rho})
\frac{P_p(\rho, u^{-1})} {P_p(\rho^*, w^{-1})}
u^{-{\frak{f}}_p ( {\rho})}
 \end{equation}
 where $L_{R}^*(E,\rho^*)$ is the leading coefficient at $s=1$ of the
$L$-function $L_R(E,\rho^*,s)$ obtained from the Hasse-Weil
$L$-function of $E$ twisted by $\rho^*$ by removing the Euler
factors at $R$.
%
Here the number ${r(E)(\rho^*)}$ is defined in \cite[(51)]{burns-ven2} (with $M=h^1(E)(1)$) and equals $\dim_{\mathbb{C}_p}(e_{\rho^*}(\mathbb{C}_p\otimes_{\z} E(K^{\ker(\rho)})))$ if  the Tate-Shafarevich group $\sha(E{/F_\infty^{\ker(\rho)}})$ is finite.

We write $X(E/F_\infty)$ for the Pontriagin dual of the ($p$-primary) Selmer group of $E$ over $F_\infty$.

\begin{thm}\label{burns-descent} Let $F$ be a number field contained in $F_\infty$ and assume
\begin{enumerate}
\item the $\mathfrak{M}_{H}(G)$-Conjecture,
\item that $\L_E$ satisfies the refined interpolation property \eqref{padicBSD},
\item that $X(E/F_\infty)$ is semisimple  at all
$\rho$ in $\mathrm{Irr}(G_{F/\q})$ in the sense of \cite[def.\ 3.11]{burns-ven1}.
\end{enumerate}
Then the `$p$-part' of the ETNC for $(E,\mathbb{Z}[G(F/\q)])$ is true. If, moreover, the Tate-Shafarevich group $\sha(E{/F})$ of $E$ over $F$ is finite,
then it   implies the `$p$-part' of a Birch and Swinnerton-Dyer type formula (see, for example,
\cite[\S 3.1]{ven-BSD}).
\end{thm}

For more details on the `$p$-part' of the Equivariant Tamagawa Number Conjecture (ETNC) and the proof of this result, which uses the existence of \eqref{appCM} as shown in this paper, see \cite[thm.\ 8.4]{burns-ven2}. Note that due to our semisimplicity assumption combined with \cite[rem.\ 7.6, prop.\ 7.8]{burns-ven2}  the formula \eqref{padicBSD} coincides with that of \cite[Conj.\ 7.4]{burns-ven2}. Also Assumption {\rm (W)} of thm.\ 8.4 is valid for weight reasons. Finally we note that by \cite[lem.\ 3.13, 6.7]{burns-ven1} $X(E/F_\infty)$ is semisimple  at $\rho$ if and only if the $p$-adic height pairing \[ h_p(V_p(E)\otimes \rho^*): H^1_f(\Q,V_p(E)\otimes \rho^*)\times H^1_f(\Q,V_p(E)\otimes \rho)\to L \] from \cite[\S 11]{nek} (see also \cite{schneider82} or \cite{perrin-height}) is non-degenerate, where $V_p(E)=\qp\otimes T_E$ is the usual $p$-adic representation attached to $E$. As far as we
are aware, the only theoretical evidence for non-degeneracy is a
result of Bertrand \cite{ber} that for an elliptic curve with
complex multiplication, the height of a point of infinite order is
non-zero. Computationally, however, there has been a lot of work done
recently by Stein and Wuthrich \cite{wuth}.

\appendix

\section{$p$-adic Hodge theory and $(\varphi,\Gamma)$-modules}

As before in the local situation $K$ denotes a (finite) unramified extension of $\qp.$ Let $\eta:
G_{\qp}\to \zp^\times$ (here $\zp^\times$ can also be replaced by $\O_L^\times,$ but for simplicity of
notation we won't do that in this exposition) be an unramified character and let $T_0$ be the free
$\zp$-module with basis $t_{ \eta,0}$ such that $\sigma\in G_{\qp}$ acts via $\sigma
t_{\eta,0}=\eta(\sigma)t_{\eta,0}.$ More generally, for $r\in \z,$ we consider the $G_{\qp}$-module
\[T:= T_0(r),\] free as a $\zp$-module with basis $t_{\eta,r}:=t_{\eta,0}\otimes \epsilon^{\otimes
r},$ where $\epsilon=(\epsilon_n)_n$ denotes a fixed generator of $\zp(1),$ i.e.,
$\epsilon_n^p=\epsilon_{n-1}$ for all $n\geq 1,$ $\epsilon_0=1$ and $\epsilon_1\neq 1.$ Thus we
have $\sigma(t_{\eta,r})=\eta(\sigma)\kappa^r(\sigma)t_{\eta,r}, $ where $\kappa:G_{\qp}\to
\zp^\times$ denotes the $p$-cyclotomic character. Setting $V:=\qp\otimes T=V_0(r)$  we obtain for
its de Rham filtration
\be\label{deRhamfilt}  D^i_{dR}(V)=\left\{
                   \begin{array}{ll}
                     D_{dR}(V)\cong Ke_{\eta,r} , & \hbox{$i\leq -r;$} \\
                     0, & \hbox{otherwise.}
                   \end{array}
                 \right.
\ee
where $e_{\eta,r}:= at^{-r}\otimes t_{\eta,r}$ with a unique $a=a_\eta\in\zpnr^\times,$ such that
$\tau_p(a)=\eta^{-1}(\tau_p)a,$ see \cite[thm.\ 1, p.\ III-31]{serre}. Here as usual $t=\log
[\epsilon] \in B_{cris}\subseteq B_{dR}$ denotes the $p$-adic period analogous to $2\pi i.$
Furthermore we have
\[D_{cris}(V)=K e_{\eta,r}\] with
\[\varphi(e_{\eta,r})=p^{-r}\eta^{-1}(\tau_p)e_{\eta,r}.\] If $\eta$ is trivial, we also write
$t_r$ and $e_r$ for $t_{\eta,r}$ and $e_{\eta,r},$ respectively.

Now consider  the $\O_K$-lattices \[M_0:=\O_Ke_{\eta,0}=(\zpnr\otimes_{\zp}T_0)^{G_K}\subseteq
D_{cris}(V_0)\] and \[M:=(t^{-r}\otimes \epsilon^{\otimes r})M_0=\O_Ke_{\eta,r}\subseteq
D_{cris}(V).\]

Using the variable $X=[\epsilon]-1$ we have $t=\log(1+X)$ and on the rings
\[\O_K[[X]]\subseteq B_{rig,K}^+:=\{f(X)=\sum_{k\geq 0}a_k X^k|a_k\in K,\; f(X)\mbox{ converges on }\{x\in\mathbb{C}_p| |x|_p<1\}\}\]
we have the following operations: $\varphi$ is induced by the  usual action of $\phi$ on the
coefficients and by $\varphi(X):=(1+X)^p-1,$ while $\gamma\in\Gamma$ acts trivially on coefficients
and by $\gamma(X)=(1+X)^{\kappa(\gamma)}-1$; letting $H_K=G(K/\qp)$ act just on the coefficients we
obtain a $\Lambda(G)$-module structure on $\O_K[[X]].$ Moreover, $\varphi$ has a left inverse operator
$\psi$ uniquely determined via $\varphi\circ \psi(f)=\frac{1}{p}\sum_{\zeta^p=1} f(\zeta(1+X)-1).$ The
differential operator $D:=(1+X)\frac{d}{dX}$ satisfies
   \be\label{D} D\varphi f=p\varphi Df \mbox{ and } D\gamma f=
    \kappa(\gamma)\gamma Df\ee
It is well-known \cite[lem.\ 1.1.6]{perrin94}  that $D$ induces an isomorphism of
$\O_K[[X]]^{\psi=0}$. Furthermore, setting
$\Delta_if:=D^if(0) $ for $f\in \O_K[[X]]^{\psi=0},$  we have an exact sequence (loc.\ cit., \S
2.2.7, (2.1)) \begin{align}
 \label{PR} \xymatrix{
  0 \ar[r] & t^r\otimes D_{cris}(V)^{\varphi=p^{-r}} \ar[r]^{ } & (B^+_{rig,K}\otimes_K D_{cris}(V))^{\psi=1} \ar[r]^(0.55){1-\varphi} & {\phantom{mmmmmmmm}}
   }\\
  \xymatrix{ (B^+_{rig,K})^{\psi=0}\otimes_KD_{cris}(V)
  {\ar[r]^{\Delta_r}} &(D_{cris}(V)/(1-p^r\varphi))(r) \ar[r] & 0
  }\notag
\end{align}
where $\varphi$ (and $\psi$) acts diagonally on $B^+_{rig,K}\otimes_K D_{cris}(V),$ while $D$ operates
just on the first tensor-factor.  We set \[\mathcal{D}_M:=\O_K[[X]]^{\psi=0}\otimes_{\O_K} M, \]  while by
 \[D(T)=(\mathbb{A}\otimes_{\zp}T)^{G_{K_\infty}}\] we denote the $(\varphi,\Gamma)$-module attached to $T,$
 where the definition of the ring $\mathbb{A}$ together with its $\varphi$- and $\Gamma$-action can be
 found e.g.\ in \cite{berger-exp}. Here we only recall that $\mathbb{A}_K^+\cong \O_K[[X]]$ and $\mathbb{A}_K\cong (\O_K[[X]][\frac{1}{X}])^{\wedge
 p\textrm{-adic}}$ is the $p$-adic completion of the Laurent series ring.

\begin{rem} \begin{enumerate}
\item Let $\eta$ be non-trivial.  From \cite[thm. A.3]{berger-exp} and its proof one sees immediately that for the Wach-module
$N(T_0)$ which according to (loc.\ cit., prop. A.1)   equals $\O_K[[X]]\otimes_{\O_K}M_0, $ the
natural inclusion $N(T_0)\hookrightarrow \mathbb{A}_K\otimes_{\mathbb{A}_K^+}N(T_0)$ induces an
isomorphism
\[\xymatrix{
 {(\O_K[[X]]\otimes_{\O_K}M_0)^{\psi=1}\ar[r]^(0.6){\cong}} & N(T_0)^{\psi=1}\ar[r]^(0.3){\cong} & \left( \mathbb{A}_K\otimes_{\mathbb{A}_K^+}N(T_0)\right)^{\psi=1}\cong D(T_0)^{\psi=1}
  }.\]
\item If $\eta$ is trivial, one has similarly
$N(\zp)=\mathbb{A}_K^+=\O_K[[X]] $ by (loc.\ cit., prop. A.1), whence
$N(\zp(1))=X^{-1}\mathbb{A}_K^+\otimes t_1=X^{-1}\O_K[[X]]\otimes t_1$ by the usual twist behaviour
of Wach modules. We obtain
\[ D(\zp(1))^{\psi=1}\cong N(\zp(1))^{\psi=1}=(X^{-1}\O_K[[X]]\otimes t_1)^{\psi=1}=\zp X^{-1}\otimes t_1\oplus ( \O_K[[X]]\otimes t_1)^{\psi=1},\] but  $N(\zp )^{\psi=1}\not\cong D(\zp
)^{\psi=1}$ according to (loc.\ cit., prop. A.3).
\end{enumerate}
\end{rem}

We define $\tilde{D}(\zp(r))^{\psi=1}=( \O_K[[X]]\otimes t_r)^{\psi=1}$ and $\tilde{D}(T)^{\psi=1}=
{D}(T)^{\psi=1}$ for non-trivial $\eta$ and obtain a canonical isomorphism
\be\label{tildeD}(\O_K[[X]]\otimes_{\O_K}M)^{\psi=p^r}\cong \tilde{D}(T)^{\psi=1}\ee
  induced by  multiplication with $t^r:$  \[f(X)\otimes
(oat^{-r}\otimes t_{\eta,r}) \mapsto f(X)oa\otimes t_{\eta,r},\]
where $o\in\O_K$ and $a$ is as before.

Setting $ \T_{K_\infty}:=\T_{K_\infty}(T):= \Lambda(G( K_\infty/\qp))^\sharp\otimes_{\zp}T$ we
recall that there is a canonical isomorphism due to Fontaine \be\label{fontaine}D(T)^{\psi=1}\cong
\H^1(\qp, \T_{K_\infty}),\ee    which e.g.\ is called $\{h^1_{K_n,V}\}_n$ in \cite{berger-exp} and
its inverse $\mathrm{Log}^*_{T^*(1)}$ in \cite[rem.\ II.1.4]{cher-col1999}.

 I am very grateful to Denis Benois for parts of the proof of the following proposition, which has been
 stated in a slightly different form, but without proof\footnote{As twisting with the cyclotomic character starting from $\Qp(1)$ only recovers the representations $V=\Qp(r),$ the general case with $V_0$ being non-trivial is not covered.}, in \cite[prop.\ 4.1.3]{perrin94}.


 \begin{prop}\label{twistingPR}
 \begin{enumerate}
 \item There is a canonical exact sequence   of $\O_K$-modules
 \[\xymatrix@C=0.7cm{
   0 \ar[r]^{ } & 1\otimes M^{\varphi=p^{-r}} \ar[r]^{ } & (\O_K[[X]]\otimes_{\O_K}M)^{\psi=p^r} \ar[r]^{ } & {\mathcal{D}_M}  \ar[r]^(0.3){\Delta_{M,r}} & M/(1-p^r\varphi)M \ar[r]^{} & 0,   }\]
   where the map in the middle is up to twisting induced by $1-\varphi,$ see the first diagram in the proof
   below.
\item Assume that $\eta$ is non-trivial. Then, using the isomorphisms \eqref{tildeD} and \eqref{fontaine}      we obtain the following commutative diagram of $\La(G)$-modules, in which the maps $\C(\T_{K_\infty})$ ($ =(D^{-r}\otimes t^{-r})(1-\varphi) $) and $\L_0(\T_{K_\infty}) $ are defined by the property that the  rows become isomorphic to the exact sequence
in (i)
\[\xymatrix{
  0 \ar[r]^{ } &   D(T)^{\varphi=1}\ar[d]^{\cong} \ar[r]^{ } & D(T)^{\psi=1} \ar[r]^(0.6){ \C(\T_{K_\infty})}\ar[d]^{\cong}_{\mathrm{Log}^*_{T^*(1)}} & {\mathcal{D}_M}  \ar[d]^{=}\ar[r]^(0.3){\Delta_{M,r}} & M/(1-p^r\varphi)M \ar[d]^{=}\ar[r]^{} & 0    \\
  0 \ar[r]^{ } &   {\H^1(\qp, \T_{K_\infty})_{tors}}  \ar[r]^{ } & {\H^1(\qp, \T_{K_\infty})} \ar[r]^(0.65){\L_0(\T_{K_\infty}) } & {\mathcal{D}_M}  \ar[r]^(0.3){\Delta_{M,r}} & M/(1-p^r\varphi)M \ar[r]^{} & 0 .     }\]
\item The sequence \eqref{localunitsfinte} can be interpreted in terms of $(\varphi,\Gamma)$-modules by the
 following commutative diagram
\[\xymatrix{
 0 \ar[r] & {\lu(K_\infty)} \ar[d]^{Dlog\; g_{-}}_{\cong} \ar[r]^(0.3){\delta } & {\H^1(\qp,\Tun)\cong \projlim n \widehat{K_n^\times}} \ar[d]_{\cong} \ar[r]^(0.8){-\hat{v}} & {\zp} \ar@{=}[d]_{ } \ar[r]^{ } & 0   \\
  0 \ar[r]^{ } & {\tilde{D}(\zp(1))^{\psi=1}} \ar[r]^{ } & { {D}(\zp(1))^{\psi=1}} \ar[r]^{ } & {\zp} \ar[r]^{ } & 0,
  }\] where $\delta$ denotes the Kummer map and $\hat{v}$ is induced from the normalized valuation map. Furthermore,
  we  obtain again a commutative diagram  of $\La(G)$-modules, in which the maps $ {\C}(\T_{K_\infty})$ ($ =(D^{-r}\otimes t^{-r})(1-\varphi) $) and $ {\L}_0(\T_{K_\infty}) $ are defined by the property that the rows become isomorphic to the exact sequence
in (i)
\[\xymatrix{
  0 \ar[r]^{ } &   {\tilde{D}(\zp(r))^{\varphi=1} \ar[r]^{ } }& {\tilde{D}(\zp(r))^{\psi=1}} \ar[r]^(0.6){  {\C}(\T_{K_\infty})} & {\mathcal{D}_M}  \ar@{=}[d] \ar[r]^(0.3){\Delta_{M,r}} & M/(1-p^r\varphi)M \ar[r]^{} & 0    \\
  0 \ar[r]^{ } &   {\zp(r) \ar[r]^{ }\ar[u]_{\cong}} & {\lu(K_\infty)(r-1)} \ar[u]^{\cong}_{Dlog\; g_{-}}  \ar[r]^(0.65){ {\L}_0(\T_{K_\infty})}   & {\mathcal{D}_M}  \ar[r]^(0.3){\Delta_{M,r}} &{\zp(r)} \ar[u]^{\cong}\ar[r]^{} & 0 .     }\]
Hence, using the map $\xymatrix{
 {\frak{M}}\otimes e_1: {\O_K[[\Gamma]] }  \ar[r]^(0.7){\cong} & {\mathcal{D}_M}   },$   $\lambda\mapsto \lambda \cdot (1+X)\otimes e_{1}$, where $\frak{M}$ denotes  the Mahler (or $p$-adic Mellin) transform  \cite[thm. 3.3.3]{co-su-buch-MC}, the lower sequence can be canonically identified with
 Coleman's exact sequence \eqref{colemanfinite}: $ {\L}_0(\T_{K_\infty})=(\frak{M}\otimes e_1)\circ Col_\epsilon.$

 \end{enumerate}
 \end{prop}

\begin{proof}
 The exactness in (i) for $M_0$ can be checked as follows. Let $f(X)\otimes e_{\eta,0}$ be in
 ${\mathcal{D}_{M_0}}^{\Delta_{M_0,0}=0},$ i.e., $f(0)e_{\eta,0}=(1-\varphi)b$ for some $b\in M_0.$
 Hence $(f(X)-f(0))\otimes e_{\eta,0}=Xg(X)\otimes e_{\eta,0}$ for some $g\in \O_K[[X]]$ and \[F':=(1-\varphi)^{-1}(Xg(X)\otimes e_{\eta,0}):=\sum_{i\geq 0}\varphi^i(Xg(X)\otimes e_{\eta,0})\in \O_K[[X]]\otimes M_0\]
is a well-defined element. Setting $F:=F'+b$ we have $(1-\varphi)F=f(X)\otimes e_{\eta,0}$ as
desired. Now exactness follows from \eqref{PR}.
 The general case follows from the following commutative ``twist diagram'' of $\O_K$-modules
 \[\xymatrix@C=0.6cm{
   0 \ar[r]^{ } & 1\otimes M^{\varphi=p^{-r}}\ar[d]^{\cong}_{1\otimes(t^r\otimes\epsilon^{\otimes -r})} \ar[r]^{ } & (\O_K[[X]]\otimes_{\O_K}M)^{\psi=p^r} \ar[d]^{\cong}_{1\otimes(t^r\otimes\epsilon^{\otimes -r})} \ar[r]^{ } & {\mathcal{D}_M} \ar[d]^{\cong}_{D^r\otimes(t^r\otimes\epsilon^{\otimes -r})} \ar[r]^(0.3){\Delta_{M,r}} & M/(1-p^r\varphi)M \ar[d]^{\cong}_{ t^r\otimes\epsilon^{\otimes -r} }\ar[r]^{} & 0 \phantom{.}   \\
   0 \ar[r]^{ } & 1\otimes M_0^{\varphi=1} \ar[r]^{ } & (\O_K[[X]]\otimes_{\O_K}M_0)^{\psi=1} \ar[r]^(0.7){ 1-\varphi} & {\mathcal{D}_{M_0}}  \ar[r]^(0.3){\Delta_{M_0,0}} & M_0/(1-p^r\varphi)M_0 \ar[r]^{} & 0.
   }\] Item (ii) is clear by the fact that $D(T)=\mathbb{A}_K\cdot a\otimes t_{\eta,r},$ which can
   either be calculated directly or deduced from the above remark. The
    statement about the torsion (first vertical isomorphism) follows from \cite[thm.\
   5.3.15]{colmez-tsinghua}.

For (iii) first note that by \cite[prop. V.3.2 (iii)]{cher-col1999} we have a commutative diagram
\[\xymatrix{
  {\mathbb{U}(K_\infty) \ar[dr]_{\delta} \ar[r]^{\Upsilon}}
                &  D(\zp(1))^{\psi=1}\ar@{=}[r] \ar[d]^{\mathrm{Log}^*_{\qp}}  & D(\zp)^{\psi=1}(1)    \\
                & {\H^1(\qp,\T_{K_\infty}(\zp(1)))}             }\]
where   $\Upsilon$ maps  $u$ to $ \frac{Dg_u}{g_u}\otimes t_1=D \log g_u \otimes t_1.$ The
statements concerning the first diagram follow easily, see also \cite[\S 7.2]{colmez-tsinghua}. The
second diagram follows as above. By construction the composite
\[\mathbb{U}(K_\infty)\to D(\zp(1))^{\psi=1}\to {\mathcal{D}_M}\] maps $u=(u_n)_n$ to \begin{eqnarray*}
\left(D^{-1}(1-\varphi) D \log g_u\right)\otimes e_1&=& \left((1-p^{-1}\varphi)\log
g_u\right)\otimes
e_1\\
&=&(1- \varphi)\left(\log g_u\otimes e_1\right)\\
&=&\mathcal{L}(g_u)\otimes e_1\\
&=&Col(u)\cdot (1+X)\otimes e_1,\end{eqnarray*} where $\mathcal{L}$ has been defined in
\eqref{colpowerseries}. This implies the last statement.
\end{proof}


 Now let $K$ be again a finite extension of degree $d_K$ over $\qp.$ For a uniform treatment we define
\[\tilde{\H}^1(\qp, \T_{K_\infty}(T)):=\left\{
    \begin{array}{ll}
      \H^1(\qp, \T_{K_\infty}(T)), & \hbox{if $\eta\neq \1$;} \\
      \lu(K_\infty) (r-1), & \hbox{if $T=\zp(r)$.}
    \end{array}
  \right.
\]

Setting ${\mathcal{H}_M}:=\{F\in B_{rig,K}^+ \otimes_{\O_K} M | (1-\varphi)f\in {\mathcal{D}_M}\},$
using \cite[thm.\ II.11]{berger-exp} and the commutativity of the following diagram
\[\xymatrix{
  {\mathcal{H}_M} \ar[d]_{D^r\otimes(t^r\otimes \epsilon^{\otimes -r})} \ar[r]^{1-\varphi} & {\phantom{mm}\mathcal{D}_{M}^{\Delta_{M,r}=0} } \ar[d]^{D^r\otimes(t^r\otimes \epsilon^{\otimes -r})} \\
  (\O_K[[X]]\otimes_{\O_K}M_0)^{\psi=1}\ar[r]^(0.6){1-\varphi} & {\phantom{mm}\mathcal{D}_{M_0}^{\Delta_{M_0,0}=0} }  }\]
we see that the map $\L_0(\T_{K_\infty})$ coincides with the ``inverse" of Perrin-Riou's  large
exponential map $\Omega_{T,r}: \mathcal{D}_{M}^{\Delta_{M,r}=0}\to D(T)^{\psi=1}/T^{H_K}(\cong
\H^1(\qp, \T_{K_\infty})/T^{H_K})$ in \cite{perrin99} (respectively $(-1)^{r-1}$ times the one in
\cite{perrin94}), which sends $f$ to $(D^r\otimes t^r)F,$ where $F\in\mathcal{H}_M$ satisfies
$(1-\varphi)F=f.$ Here ``$D^r\otimes t^r$'' denotes the composite
\[\xymatrix{
  {\mathcal{H}_M} \ar[r]^(0.3){D^r\otimes(t^r\otimes \epsilon^{\otimes -r})} & (\O_K[[X]]\otimes_{\O_K}M_0)^{\psi=1} \ar[rr]^{1\otimes(t^{-r}\otimes \epsilon^{\otimes r})} && (\O_K[[X]]\otimes_{\O_K}M)^{\psi=p^r} \ar[r]^(0.7){t^r} & D(T)^{\psi=1}  }\]
and corresponds to the operator $\nabla_{r-1}\circ \ldots \circ \nabla_0$ in \cite{berger-exp} for
$r\geq 1.$ In particular, by \cite[thm.~ II.10/13]{berger-exp} we obtain the following descent diagram for
$r,n\geq 1$, where the maps $\Xi_{M,n}=\Xi_{M,n}^\epsilon$ are recalled below \eqref{Xi} \be\label{descentberger}\xymatrix{
  {\tilde{\H}^1(\qp,\T_{K_\infty}(T))}  \ar[dd]_{pr_n} \ar[rr]^(0.6){ {\L}_0(\T_{K_\infty}(T)) }  & & {\mathcal{D}_M} \ar[dd]^{\Xi_{M,n}}   \\
    &   & \\
     {\H^1(K_n, V) }   & & K_n\cong D_{dR,K_n}(V)   \ar[ll]_(0.6){(-1)^{r-1}(r-1)!exp_{K_n,V}},   } \ee
while for $r\leq 0$ \be\label{descentberger2}\xymatrix{
  {\tilde{\H}^1(\qp,\T_{K_\infty}(T))}  \ar[dd]_{pr_n} \ar[rr]^(0.6){ {\L}_0(\T_{K_\infty}(T)) }  & & {\mathcal{D}_M} \ar[dd]^{\Xi_{M,n}}   \\
    &   & \\
     {\H^1(K_n, V) }  \ar[rr]^(0.4){ (-r)!exp_{K_n,V^*(1)}^*} & & K_n\cong D_{dR,K_n}(V).      } \ee

\begin{rem}
In particular, for $T=\zp(1)$   we have the following commutative descent diagram for $n\geq 1$
\[\xymatrix{
 {\mathbb{U}(K_\infty)} \ar[ddd]_{pr_n}  \ar[rr]^{ {\L}_0(\T_{K_\infty}(\zp(1))) } &   &  {\mathcal{D}_M } \ar[ddd]^{\Xi_n} \\
 &   &   \\
    &   &  \\
  {\qp\otimes_{\zp}U_n}   \ar[rd]^{\delta} &     & K_n\cong D_{dR,K_n}(\qp(1))\ar[dl]_{exp_{K_n,\qp(1)}}\ar[ll]_{exp },  \\
&  {\H^1(K_n, \qp(1))}  &    }\]
   where $exp$ denotes the usual $p$-adic exponential (series), while $\Xi_n$ maps
   $((1-p^{-1}\varphi)\log g_u)\otimes e_1$ to $\log g_u^{\phi^{-n}}(\epsilon_n-1)  =\log
   u_n  .$
\end{rem}

 In order to arrive at a
morphism \[ {\L}(\T_{K_\infty}(T)):\tilde{\H}^1(\qp, \T_{K_\infty}(T))\to
\T_{K_\infty}(T)\otimes_{\La}\La_{[\T(T),\tau_p]^{-1}},\] where
 $[\T,\tau_p]^{-1}=\tau_p\eta^{-1}(\tau_p),$ generalising $\mathcal{L}_{K,\epsilon}$ in \eqref{col2}, we compose $ {\L}_0(\T_{K_\infty}(T))$ with the following
canonical isomorphisms \be\label{DMiso}\xymatrix{
  {\mathcal{D}_{M}=\O_K[[X]]^{\psi=0}\otimes_{\O_K} M }& {\O_K[[\Gamma]]\otimes_{\O_K}M} \ar[l]^(0.4){\cong}_(0.4){\Psi_M} \ar[r]_{\cong}^{\Theta_M} & {\T_{K_\infty}\otimes_{\La}\La_{[\T,\tau_p]^{-1}}} ,
  }\ee
where the left one   $\Psi_M(\lambda\otimes m)= \lambda\cdot (1+X)\otimes m$ is induced by
$\frak{M},$ while the right one is given by
\begin{eqnarray*}
\Theta_M(\lambda\otimes (at^{-r}\otimes t_{ \eta,r}))&=&(1\otimes t_{ \eta,r})\otimes
(\sum_{i=0}^{d_K-1}
\tau_p^i\otimes\eta^{-i}(\tau_p)\phi^{-i}(\lambda a))\\
&=&(1\otimes t_{ \eta,r})\otimes (\sum_i \tau_p^i\otimes \phi^{-i}(\lambda )a).
\end{eqnarray*}


Similarly  to the original Coleman map $Col$ in \eqref{colemanfinite} the homomorphisms
  $ {\mathcal{C}}(\T_{K_\infty}),$  $
 {\mathcal{L}}_0(\T_{K_\infty})$ and $ {\mathcal{L}}(\T_{K_\infty})$ are norm compatible
when enlarging $K$ within $\qp^{ur}.$ Thus, by taking inverse limits we may and do define them also
for infinite unramified extensions $K$ of $\qp.$ Then we have the following twist and descent
properties:

\begin{lem}\label{twistL} Let $K'\subseteq K$ be (possibly infinite) unramified extensions of $\qp$ and $Y$ a  $(\La(G(K'_\infty/\qp)),\La(G(K_\infty/\qp))$-module such that
$Y\otimes_{\La(G(K_\infty/\qp))}\T_{K_\infty}(T)\cong \T_{K'_\infty}(T')$ as
$\La(G(K'_\infty/\qp))$-module with compatible $G_{\qp}$-action. Then
   \[Y\otimes_{\La(G(K_\infty/\qp))} {\mathcal{L}}_0(\T_{K_\infty}(T))= {\mathcal{L}}_0(\T_{K'_\infty}(T'))\] and
 \[Y\otimes_{\La(G(K_\infty/\qp))} {\mathcal{L}}(\T_{K_\infty}(T))= {\mathcal{L}}(\T_{K'_\infty}(T')).\]
 In  particular, $ {\L}(\T_{K'_\infty}(T))= {\L}_{\T_{K'_\infty}(T),\epsilon}$ in \eqref{LT}.
\end{lem}

\begin{proof}
The proof can be parted into a twist-statement, where $K'=K$ and $T'\cong T\otimes_{\zp}T'', $ such
that $G_{\qp}$ acts diagonally on the tensor product and $T'$ is rank one $\zp$-representation of
$G,$ and a descent statement. One first proves the twist-statement for $T''/p^n,$ $n$ fix, and all
finite subextensions $K'$ of $K,$ such that $G(K/{K'})$ acts trivially on $T''/p^n.$ Afterwards one
takes limits over $K'$ obtaining the twist-statement for $T''/p^n.$ Then, taking the projective
limit with respect to $n$ (see \cite{berger-limits} for the correct behaviour of
$(\varphi,\Gamma)$-modules under such limits) one shows the full twist-statement (compare with the
well-known twisting for $\H^i_{IW}$). The descent-statement then follows easily from the
 norm-compatibility and the fact that the twisted analogue of the exact sequence \eqref{colemaninfinite}
 \[\xymatrix@C=0.5cm{
   0 \ar[r] & {\tilde{\H}^1(\qp, \T_{K_\infty}(T))}\ar[rr]^(0.45){ {\mathcal{L}}(\T_{K_\infty})} && {\T_{K_\infty}(T)\otimes_{\La}\La_{[\T(T),\tau_p]^{-1}}} \ar[rr]^{ } && T   \ar[r] & 0 }\]
recovers (for finite extension $K'$ of $\qp$) the exact sequence
\[\xymatrix@C=0.5cm{
   0 \ar[r] & T^{G(K/K')} \ar[r]& {\tilde{\H}^1(\qp, \T_{K'_\infty}(T))}\ar[rr]^(0.45){ {\mathcal{L}}(\T_{K'_\infty})} && {\T_{K'_\infty}(T)\otimes_{\La}\La_{[\T(T),\tau_p]^{-1}}} \ar[r]^{ } & T   \ar[r] & 0 }\]
by taking $G(K/K')$-coinvariants.
 In more detail the unramified twist (the cyclotomic twist being well-known): Assume that $\eta$ factorises over $G(K/\qp),$ i.e., $a=a_\eta\in \O_K^\times,$ and let
 $N:= \O_Ke_{ r}\subseteq D_{cris}(\qp(r))$ be the lattice associated to $\qp(r).$ Then we have the
 following commutative diagram of $\Lambda$-modules
\[\xymatrix{
  {\left(\O_K[[X]]^{\psi=0}\otimes_{\O_K}N\right) \otimes_{\zp}T_0}\ar[d]_{a^{-1}\otimes a\otimes 1}  & {\left(\O_K[[\Gamma]]\otimes N\right)\otimes_{\zp}T_0 } \ar[d]_{a^{-1}\otimes a\otimes 1}\ar[l]_(0.45){\Psi_N\otimes T_0 } \ar[r]^{\Theta_N\otimes T_0} & {\Lambda\otimes_{\Lambda,f}\T(\zp(r))\otimes \Lambda_{\tau_p} \ar[d]^{\vartheta\otimes \tilde{f}} }\\
  {\O_K[[X]]^{\psi=0}\otimes_{\O_K}M}  & { \O_K[[\Gamma]]\otimes M} \ar[r]^{\Theta_M}\ar[l]_{\Psi_M } & {\T(T)\otimes_\Lambda \Lambda _{\tau_p\eta(\tau_p)^{-1}}}, }  \]
  where in the top line the $\Lambda$-action is induced by the diagonal $G$-action and via left multiplication on $\Lambda,$ respectively,
\[\Theta_N\otimes T_0(\lambda\otimes  (t^{-r}\otimes t_{  r})\otimes t_{ \eta,0})= 1\otimes 1\otimes
t_r\otimes\sum_i \tau^i_p\otimes\phi^{-i}(\lambda)  \]
 and $\tilde{f}:=f\otimes 1$ on
$\Lambda\widehat{\otimes}\zpnr$ is induced by $f:\Lambda\to\Lambda,$ $g\mapsto \eta(g)^{-1}g,$ while
\[\vartheta:\Lambda\otimes_{\Lambda,f}\T(\zp(r))\to \T(T), \;\; a\otimes (b\otimes t_r)\mapsto
af(b)\otimes t_{\eta,r}.\] Here $\Lambda\otimes_{\Lambda,f}-$ indicates that the tensor product is
formed with respect to $f.$ Also we have the commutative diagram
\[\xymatrix{
  D(\zp(r))^{\psi=1}\otimes T_0 \ar[d]_{\cong} \ar[rr]^(0.43){\C(\T_{K_\infty}(\zp(r)))} && {\left(\O_K[[X]]^{\psi=0}\otimes_{\O_K}N\right) \otimes_{\zp}T_0} \ar[d]^{a^{-1}} \\
  D(T)^{\psi=1} \ar[rr]^{\C(\T_{K_\infty}(T))} &&  {\O_K[[X]]^{\psi=0}\otimes_{\O_K}M.}   }\]
\end{proof}

As in section \ref{descent} we set $\La'=\mathbb{Q}_p[G_n].$

\begin{lem} \label{epsilondR} There are natural isomorphisms
\begin{enumerate}
\item $\Sigma_{M,n}: K_n'\otimes M= K'_n(a  t^{-r}\otimes t_{r,\eta})\cong D_{dR}(V')  $  of $\La'$-modules,
\item $1\otimes\Sigma_{M,n}: V_{\rho^*}\otimes_{\La'}K'_n\otimes M\cong V_{\rho^*}\otimes_{\La'}D_{dR}(V')\cong D_{dR}(W)  $ of $L$-vector spaces,
\end{enumerate}
\end{lem}

\begin{proof}
The canonical isomorphism (which makes explicit the general formula  $(Ind^H_G (B\otimes V))
\cong(B\otimes \mathrm{Ind} V) $)
\[ \qp[G_{\qp}]\otimes_{\qp[G_{K_n'}]}\left(B_{dR}\otimes_{\qp}\qp(\eta)(r)\right)  \cong         B_{dR}\otimes_{\qp} \qp[G_n]^\sharp\otimes_{\qp} \qp(\eta)(r) ,\]
which maps $g\otimes a\otimes b$ to $ga\otimes \bar{g}^{-1}\otimes gb$ with $g\in G_{\qp}$ induces
the isomorphism (via the general isomorphism $ N^H\cong (Ind^H_G N)^G, n\mapsto \sum_{\bar{g}\in
G/H} g\otimes n$\big)
\[K_n' \cdot(at^{-r}\otimes t_{r,\eta})=\left(B_{dR}\otimes\qp(\eta)(r)\right)^{G_{K_n'}}\cong
D_{dR}(V'),\] which maps $x\cdot at^{-r}\otimes t_{r,\eta}$ to
\[\sum_{g\in G_n}  g(xat^{-r})\otimes g^{-1}\otimes gt_{r,\eta} =\sum_{g\in G_n} g(x) at^{-r}\otimes g^{-1}\otimes t_{r,\eta} .    \]

Putting $e_{\eta,r}:=at^{-r}\otimes t_{r,\eta}$  we similarly obtain the isomorphism in (ii)
sending
\[l\otimes x\otimes e_{\eta,r} \mapsto \sum_{g\in G_n} g(x) at^{-r}\otimes \rho(g)l\otimes
t_{r,\eta},\]
where this element is regarded in $B_{dR}\otimes_{\qp} W=B_{dR}\otimes_{\qp} L \otimes_{\qp}
\qp(\eta)(r).$ Alternatively we can read it in $(B_{dR}\otimes_{\qp} L)\otimes_L W$ as
\be\label{Sigma} \#G_n at^{-r}e_{\rho^*} (x)l\otimes t_{\rho\eta,r}.\ee
\end{proof}

Any embedding $\sigma:L_\rho\to \overline{\qp}$ induces a map
$A_\rho:=\qp^{nr}\otimes_{\qp}L_\rho\to\overline{\qp},$ $x\otimes y\mapsto x\sigma(y),$ also called
$\sigma.$

Consider the Weil group $W(\overline{\qp}/\qp),$ which fits into a short exact sequence
\[\xymatrix@C=0.5cm{
  1 \ar[r] & I \ar[rr]^{ } && W(\overline{\qp}/\qp)\ar[rr]^{v} && \z \ar[r] & 0 ,}\]
and let $D$ be the linearised $W(\overline{\qp}/\qp)$-module associated to $D_{pst}(W)=A_\rho
e_{\eta,r}(\rho),$ i.e., $g\in W(\overline{\qp}/\qp)$ acts as $g_{\mathrm{old}}\varphi^{-v(g)}$ or
explicitly via the character \[\chi_D(g):=
{   \rho(g)\eta(\tau_p)^{v(g)}p^{rv(g)}}.
\]

For an embedding $\sigma$ we write $\bar{D}_\sigma:=\overline{\qp}\otimes_{A_\rho,\sigma}
D\cong\overline{\qp}e_{\eta,r}(\rho^\sigma),$ where $\sigma$ acts coefficient-wise on $\rho.$ If
$n\geq 0$ is minimal, such that $G(\overline{\qp}/\qp^{nr}(\mu(p^n))$ acts trivial on
$\bar{D}_\sigma,$ then by property (3) and (7)\footnote{Apparently, the formula in (7) in (loc.\
cit.) is not compatible with Deligne as claimed: Deligne identifies
$W(\overline{\mathbb{Q}_p}/\qp)$ via class field theory with $\mathbb{Q}_p^\times$ by sending the
{\it geometric} Frobenius automorphism to $p,$ which induces by a standard calculation applied to
definition (3.4.3.2) for epsilon constants of quasi-characters of $\mathbb{Q}_p^\times$ in
\cite{del-localconstant} (see e.g. \cite[\S 8.5 between (4a) and (4b)]{hida93}) the formula
\[\epsilon(V_\chi,\psi)=\chi(\tau_p)^{-n}\sum_{\sigma\in \Gamma_n}\chi(\sigma)^{-1}\sigma\epsilon_n,\] while in (loc.\ cit.) the factor is just
$\chi(\tau_p)^n.$ Here $\chi:W(\overline{\mathbb{Q}_p}/\qp)\to E^\times$ is a character which gives
the action on the $E$-vector space $V_\chi.$} in \cite[\S 3.2.2]{fukaya-kato} we obtain for the
epsilon constant attached to $\bar{D}_\sigma$ (see (loc.\ cit.))
\[\epsilon(\bar{D}_\sigma,-\psi)=1,\] if $n=0,$ while for $n\geq 1$
\begin{eqnarray*}
\epsilon(\bar{D}_\sigma,-\psi)
&=&\epsilon(\bar{ D}_\sigma^*(1),\psi)^{-1}\\
&=&\left((\rho^\sigma \eta(\tau_p)p^{r-1})^n\sum_{\gamma\in\Gamma_n} \rho^\sigma(\gamma)
\gamma\cdot
\epsilon_n \right)^{-1}\\
&=&\left((\rho^\sigma \eta(\tau_p)p^{r-1})^n  \tau(\rho^\sigma,\epsilon_n)\right)^{-1}.
\end{eqnarray*}
Here $\Gamma_n:=G(K_n/K),$ $\psi:\qp\to\bar{\mathbb{Q}_p}^\times$ corresponds to the compatible
system $(\epsilon_n)_n,$ i.e.\ $\psi(\frac{1}{p^n})=\epsilon_n,$ and $\bar{D}_\sigma^*(1)$ denotes
the linearised Kummer dual of $\bar{D}_\sigma,$ i.e., \[
\chi_{\bar{D}_\sigma^*(1)}(g)=
{  \rho^\sigma(g)^{-1}\eta(\tau_p)^{- v(g)}p^{ -(r-1)v(g),}      }
\]
while
\[\tau(\rho^\sigma,\epsilon_n):= \sum_{\gamma\in\Gamma_n}\rho^\sigma(\gamma) \gamma\cdot
\epsilon_n=\#\Gamma_n e^{\Gamma_n}_{\rho^*}\epsilon_n\] denotes the usual Gauss sum. Furthermore
\[\epsilon_L( D ,-\psi)= \left( \epsilon(\bar{D}_\sigma,-\psi)\right)_\sigma \in \prod_\sigma \mathbb{Q}_p^\times\cong (\overline{\qp}\otimes_{\qp} L)^\times\subseteq (B_{dR}\otimes_{\qp}L)^\times\]
is the $\epsilon$-element as defined in \cite[\S 3.3.4]{fukaya-kato}. We may assume that $L$
contains $\qp(\mu_{p^n});$ then $\epsilon_L( D ,-\psi)$ can be identified with
 \[1\otimes
  (\rho \eta(\tau_p)p^{r-1})^{-n} \tau(\rho ,\epsilon_n)^{-1}.\]

 Hence the comparison-isomorphism  renormalised by $\epsilon_L( {D},-\psi)$
 \[\epsilon_{L,-\epsilon,dR}(W)^{-1}: W\otimes L_{[W,\tau_p^{-1}]}\to D_{dR}(W)\subseteq B_{dR}\otimes_{\qp} L\otimes_L W,\]
 is explicitly given as \be\label{edR} x\otimes l \mapsto
\epsilon_L( D ,-\psi)^{-1}(-t)^{r}l \otimes x= (-1)^r(\rho \eta(\tau_p)p^{r-1})^{ n} \tau(\rho
,\epsilon _n)t^rl\otimes x,\ee where $\epsilon_L( D ,-\psi)^{-1}(-t)^{r}l $ is considered as an
element of $B_{dR}\otimes_{\qp}L.$

In order to deduce the descent diagram \eqref{descentdiag} from \eqref{descentberger}, for $n\geq
1,$ we have to add a commutative diagram of the following form
\[\xymatrix{
  {\mathcal{D}_{M}} \ar[d]_{\Xi_{M,n}} \ar[r]^{ } & {\T_{K_\infty}\otimes_{\La}\La_{[\T,\tau_p]^{-1}}} \ar[d]^{Y\otimes_\Lambda- } \\
 K_n\otimes M\cong D_{dR,K_n}(\qp(\eta)(r))/D_{cris}(\qp(\eta)(r))^{\varphi=1} &  V'\otimes_{\La'} (\La')_{[V',\tau_p^{-1}]},   \ar[l]^{ } }\]

where
\be\label{Xi}\Xi_{M,n}(f)=\Xi_{M,n}^\epsilon(f)=p^{-n}(\phi \otimes\varphi)^{-n}\left(F\right)(\epsilon_n-1)=p^{-n}(\varphi\otimes\varphi)^{-n}\left(F\right)(0)\ee
with  $F\in {\mathcal{H}_M}$ such that  $(1-\varphi)F=f=\tilde{f}\otimes e_{\eta,r}$ (recall that
$\varphi$ acts as $\varphi\otimes \varphi$ here) on ${\mathcal{D}_{M}}^{\Delta=0}$ and more
generally $\mathrm{mod} D_{cris}(\qp(\eta)(r))^{\varphi=1}$ (recall that
$D_{cris}(\qp(\eta)(r))^{\varphi=1}=0$ in the generic case)
\begin{eqnarray*}
\Xi_{M,n}(f)&=&p^{-n}\left( \sum_{k=1}^{n}(\phi\otimes\varphi)^{-k}\left(f(\epsilon_k-1)\right) +
(1-\phi\otimes\varphi)^{-1}\left(f(0)\right) \right)  \\
&=&p^{-n}\left( \sum_{k=1}^{n}p^{kr}\eta(\tau_p)^k \tilde{f}^{\phi^{-k}}(\epsilon_k-1)  + (1-
p^{-r}\eta(\tau_p)^{-1}\phi)^{-1}\tilde{f}(0) \right) \otimes e_{\eta,r}
\end{eqnarray*}
(see \cite[Lem. 4.9]{benois-berger},   where $f(0)$ is considered in $D_{cris}(V)$ and hence the
last summand above equals $(1-\varphi)^{-1}f(0)$ there   by the $\phi$-linearity of $\varphi.$).
Here, for any $H(X)=\tilde{H}(X)\otimes e\in B^+_{rig,K}\otimes_{\O_K} M$ we consider
$H(\epsilon_k-1)=\tilde{H}(\epsilon_k-1)\otimes e,$ $k\leq n,$ as element in $K_n\otimes_{\O_K}M,$
on which $\phi\otimes \varphi$ acts naturally.

First we note that for $n\geq 1$  we have a commutative diagram \be\label{decentpsiM}\xymatrix{
  {\mathcal{D}_{M}} \ar[d]^{\Xi_{M,n}}& {\O_K[[\Gamma]]\otimes M }\ar[d]_{pr_n\otimes id} \ar[l]_{\Psi_M }\ar[r]^{\Theta_M} & {\T_{K_\infty}\otimes_{\La}\La_{[\T,\tau_p]^{-1}}}\ar[d]^{Y\otimes_\Lambda-} \\
 K_n\otimes_{\O_K} M/D_{cris}(\qp(\eta)(r))^{\varphi=1}&  K[\Gamma_n]\otimes M \ar[l]_(0.3){\Psi_{M,n}}\ar[r] &  V'\otimes_{\La'} (\La')_{[V',\tau_p^{-1}]},
 }\ee
where
\begin{eqnarray}\label{PsiM}
 \Psi_{M,n}(\mu\otimes e_{\eta,r})&=&\Psi_{M,n}^\epsilon(\mu\otimes e_{\eta,r})\notag \\
 &=&p^{ -n}\left(
\sum_{k=1}^{n}  \epsilon_k ^{ \phi^{-k}(\mu)}\otimes \varphi^{-k}(e_{\eta,r})
+(1-\phi\otimes \varphi )^{-1}( 1^\mu\otimes e_{\eta,r}) \right) \\
 &=& \left(
\sum_{k=1}^{n}p^{kr-n}\eta(\tau_p)^k \epsilon_k ^{ \phi^{-k}(\mu)}  +p^{ -n}(1-
p^{-r}\eta(\tau_p)^{-1}\phi)^{-1}( 1^\mu) \right)\otimes e_{\eta,r}\notag
\end{eqnarray}
modulo $  D_{cris}(\qp(\eta)(r))^{\varphi=1}.$ Here $\phi$ acts coefficient-wise on $K[\Gamma_n]$
and $1^\mu$ is the same as the image of $\mu$ under the augmentation map $\O_K[\Gamma_n]\to \O_K.$

\begin{prop}\label{commutative}
\begin{enumerate}
\item For   $n\geq\max\{ 1,  a(\rho)\}$ and $W\neq \qp(1),$ the following diagram is commutative:
 \[\footnotesize\xymatrix@C=0.5cm{
  V_{\rho^*}\otimes_{\qp[G_n]}K_n\otimes M \ar[d]_{1\otimes\Sigma_{M,n}}  & V_{\rho^*}\otimes_{\qp[G_n]} K[\Gamma_n]\otimes M   \ar[l]_{1\otimes \Psi_{M,n}}\ar[r]^(0.44){1\otimes\Theta_{M,n}} &  V_{\rho^*}\otimes_{\qp[G_n]} V' \otimes_{\qp[G_n]} \Lambda'_{[V',\tau_p^{-1}]}\ar[d]^{\cong} \\
  V_{\rho^*}\otimes_{\qp[G_n]} D_{dR}(V')\cong D_{dR}(W)  & D_{dR}(W)\ar[l]_(0.3){\Phi_W}  & W\otimes_L L_{[W,\tau_p^{-1}]}\ar[l]_{(-1)^r\epsilon_{L,-\epsilon,dR}(W)^{-1}} ,  }\]
  where  \[\Phi_W:=\left\{
             \begin{array}{ll}
               \mathrm{id}_{D_{dR}(W)}, & \hbox{if $a(\rho)\neq 0;$} \\
              \frac{\det(1-\varphi|D_{cris}(W^*(1)))}{\det(1-\varphi|D_{cris}(W))}, & \hbox{otherwise.}
             \end{array}
           \right.
  \]
\item For $W\neq \qp(1)$ the diagram \eqref{descentdiag} commutes.
\end{enumerate}
\end{prop}

\begin{proof}
Let $b$ denote a normal basis of $\O_K,$ i.e., $\O_K=\zp[\bar{H}]b$ with $\bar{H}=G(K/\qp),$ which
can be lifted from the residue field, $K$ being unramified, and $e:=e_{\eta,r}.$ Then $1\otimes
b\otimes e=1\otimes e_{\rho^*}b\otimes e$ is a basis of $V_{\rho^*}\otimes_{\qp[G_n]}
K[\Gamma_n]\otimes M $ as $L$-vector space (in general $\rho(g)$ does not lie in $K,$ but using
$V_{\rho^*}\otimes_{\qp[G_n]} K[\Gamma_n]\cong V_{\rho^*}\otimes_{L[G_n]} L[G_n] \otimes_{\qp[G_n]}
K[\Gamma_n]\cong V_{\rho^*}\otimes_{L[G_n]}(L\otimes_{\qp} K[\Gamma_n])$ one can make sense of it).
We calculate (going clockwise in the above diagram)
\begin{align*}
& 1\otimes\Theta_{M,n}(1\otimes b\otimes e) \\
&\qquad =  1\otimes (1\otimes t_{\eta,r})\otimes
\sum_{i=0}^{d_K-1} \tau_p^i\otimes \phi^{-i}(b)a\;\;\;\;\;\;(\subseteq V_{\rho^*}\otimes_{\qp[G_n]} V' \otimes_{\qp[G_n]} \Lambda'_{[V',\tau_p^{-1}]})\\
&\qquad =  t_{\rho\eta,r}\otimes\sum_{i=0}^{d_K-1} \rho(\tau_p)^{-i}\rho^*(\phi^{-i}(b))a \qquad\;\;\;(\subseteq W\otimes_L L_{[W,\tau_p^{-1}]})\\
&\qquad =  t_{\rho\eta,r}\otimes\sum_{i=0}^{d_K-1} \rho(\tau_p)^{-i}\phi^{-i}(b) a\\
&\qquad =  t_{\rho\eta,r}\otimes  \varsigma(\rho,b)a\\
\end{align*}
with \[\varsigma(\rho,b):= \sum_{i=0}^{d_K-1} \rho(\tau_p)^{-i}\phi^{-i}(b)=d_K
e^{\bar{H}}_{\rho^*} b\] a Gauss-like sum, where
$e^{\bar{H}}_{\rho^*}=\frac{1}{\#\bar{H}}\sum_{h\in H} \rho(h)h.$ This element is sent by
$(-1)^r\epsilon_{L,-\epsilon,dR}(W)$ to \be\label{clock}(-1)^r\epsilon_L( D ,-\psi)^{-1}(-t)^{-r}
\varsigma(\rho,b)a\otimes t_{\rho\eta,r}=   p^{mr-m}(\rho\eta)(\tau_p^{ m}) \tau(\rho
,\epsilon_m)\varsigma(\rho,b)at^{-r}\otimes t_{\rho\eta,r}\ee in $D_{dR}(W),$ where we used
\eqref{edR} with $m=a(\rho)$.

Now we determine the image of $1\otimes  b\otimes e=1\otimes e_{\rho^*}b\otimes e$ anti-clockwise.
First note that the idempotent $e_{\rho^*}$ decomposes as $e^{\Gamma_n}_{\rho^*}\cdot e^{\bar{H}
}_{\rho^*}.$

Hence, for $n\geq a(\rho)\geq1, $ where $p^{a(\rho)}$ denotes the conductor of $\rho$ restricted to
$\Gamma_n,$ we have
\begin{eqnarray*}
(1\otimes \Psi_{M,n})(1\otimes b\otimes e)&=& 1\otimes e_{\rho^*}\Psi_{M,n}(  b\otimes
e)\\
&=&1\otimes p^{nr-n}\eta(\tau_p)^n \phi^{-n}(e^{\bar{H}}_{\rho^*} b) e_{\rho^*}^{\Gamma_n}\cdot\epsilon_n \otimes e\\
&=&1\otimes p^{nr-n}\eta(\tau_p)^n {\rho^*}(\tau_p^{-n})e^{\bar{H}}_{\rho^*} b e_{\rho^*}^{\Gamma_n}\cdot\epsilon_n\otimes e\\
&=&1\otimes \frac{p^{nr-n}}{\#G_n}(\rho \eta)(\tau_p^n) \varsigma(\rho,b)\tau(\rho ,\epsilon_n
)\otimes e,
\end{eqnarray*}

where we use  the explicit formula \eqref{PsiM} and the well-known fact about Gauss sums (see e.g.\
\cite[lem. 5.2]{bf4})
\[ e_\rho^{\Gamma_n}(\epsilon_k)=\left\{
                                 \begin{array}{ll}
                                   e_\rho^{\Gamma_n}(\epsilon_k), & \hbox{if $a(\rho)=k $; } \\
                                    (1-p)^{-1}, & \hbox{if $a(\rho)=0$ and $k= 1$;} \\
                                   0, & \hbox{otherwise,}
                                 \end{array}
                               \right.
\]
where we assume $k\leq n.$ Now from \eqref{Sigma}  we see that $\Sigma_{M,n}$ sends this element,
which already ``lies in the right eigenspace" to
\[  at^{-r}p^{nr-n}(\rho \eta)(\tau_p^n) \tau(\rho
,\epsilon_n )\varsigma(\rho,b)\otimes t_{\rho\eta,r}= p^{nr-n}(\rho \eta)(\tau_p^n) \tau(\rho
,\epsilon_n )\varsigma(\rho,b)at^{-r}\otimes t_{\rho\eta,r}, \] i.e., to the same element as in
\eqref{clock}, whence the result follows, if $a(\rho)\neq 0.$

Now assume that $a(\rho)=0,$ i.e., $\rho|\Gamma_n,$ the restriction to $\Gamma_n,$ is trivial. Setting $n=1$ we then have
\begin{eqnarray*}
(1\otimes \Psi_{M,1})(1\otimes b\otimes e)&=& 1\otimes \Psi_{M,1}( e_{\rho^*} b\otimes
e)\\
&=&1\otimes \left( p^{r-1}\eta(\tau_p)  \epsilon_1^{\phi^{-1}(e_{\rho^*} b )} + p^{ -1}(1-
p^{-r}\eta(\tau_p)^{-1}\phi)^{-1}(  e_{\rho^*}^{\bar{H}} b)\right)\otimes e\\
&=&1\otimes \left( p^{r-1}\eta(\tau_p) \phi^{-1}(e^{\bar{H}}_{\rho^*} b)
e_{\rho^*}^{\Gamma_1}\cdot\epsilon_1  + p^{ -1}(1-
p^{-r}\rho\eta(\tau_p)^{-1})^{-1}(  e_{\rho^*}^{\bar{H}} b)\right)\otimes e\\
&=&1\otimes \left( p^{r-1}\rho\eta(\tau_p)    (1-p)^{-1}  + p^{ -1}(1-
p^{-r}\rho\eta(\tau_p)^{-1})^{-1} \right) \frac{\varsigma(\rho,b)}{d_K}\otimes e\\
&=&1\otimes \left( \frac{1-p^{r-1}\rho\eta(\tau_p)}{ 1-p^{-r}\rho\eta(\tau_p^{-1})}\right)\frac{\varsigma(\rho,b)}{d_K(p-1)}\otimes e,\\
  \end{eqnarray*}
which is sent under $\Sigma_{M,1}$ to
\[\left( \frac{\det(1-\varphi|D_{cris}(W^*(1)))}{\det(1-\varphi|D_{cris}(W))}\right) \varsigma(\rho,b) at^{-r}\otimes t_{\rho\eta,r}, \]
 while \eqref{clock} becomes just
\[ \varsigma(\rho,b)  at^{-r}\otimes t_{\rho\eta,r} .\]
Upon replacing $\epsilon$ by $-\epsilon=\epsilon^{-1}$ (we have used both the additive and
multiplicative notation!) the second statement follows from the diagrams \eqref{descentberger},
\eqref{decentpsiM} and the one in (i) diagram combined with the isomorphism \eqref{DMiso} and Lemma
\eqref{twistL}.
\end{proof}

\section{Determinant functors}\label{determinants} In this appendix we recall some details of the formalism of determinant functors
 introduced by Fukaya and Kato in \cite{fukaya-kato}   (see also \cite{ven-BSD}).

We fix an associative unital noetherian ring $R$. We write ${ B}(R)$ for the category of bounded
complexes of (left) $R$-modules, ${ C}(R)$ for the  category of bounded complexes of finitely
generated (left) $R$-modules, $P(R)$ for the category of finitely generated projective (left)
$R$-modules, $ {C}^{\rm p}(R)$ for the category of bounded (cohomological) complexes of finitely
generated projective (left) $R$-modules. By $D^{\rm p}(R)$ we denote the category of perfect
complexes as full triangulated subcategory of the derived category $D^{\rm b}(R)$ of $B(R).$ We
write $(C^{\rm p}(R),{\rm quasi})$ and $(D^{\rm p}(R),{\rm is})$ for the subcategory of
quasi-isomorphisms of $C^{\rm p}(R)$ and isomorphisms of $D^{\rm p}(R),$ respectively.

For each complex $C = (C^\bullet,d_C^\bullet)$ and each integer $r$ we define the $r$-fold shift
$C[r]$ of $C$ by setting $C[r]^i= C^{i+r}$ and $d^i_{C[r]}=(-1)^rd^{i+r}_C$ for each integer $i$.

We first recall that there exists a Picard category $\C_R$ and a determinant functor $\,\d_R:(
{C}^{\rm p}(R),{\rm quasi})\to \C_R$ with the following properties (for objects $C,C'$ and $C''$ of
$\mathrm{C}^{\rm p}(R)$)

\begin{itemize}
\item[B.a)] $\C_R$ has an associative and commutative product
structure $(M,N) \mapsto M\cdot N$ (which we often write more simply as $MN$) with canonical unit
object ${\bf 1}_R = \d_R(0)$. If $P$ is any object of $P(R)$, then in $\C_R$ the object $\d_R(P)$
has a canonical inverse $\d_R(P)^{-1}$. Every object of $\C_R$ is of the form $\d_R(P)\cdot
\d_R(Q)^{-1}$ for suitable objects $P$ and $Q$ of $P(R)$.

\item[B.b)] All morphisms in $\C_R$ are isomorphisms and elements of
the form $\d_R(P)$ and $\d_R(Q)$ are isomorphic in $\C_R$ if and only if $P$ and $Q$ correspond to
the same element of the Grothendieck group $K_0(R)$. There is a natural identification ${\rm
Aut}_{\C_R}({\bf 1}_R) \cong K_1(R)$ and if ${\rm Mor}_{\C_R}(M,N)$ is non-empty, then it is a
$K_1(R)$-torsor where each element $\alpha$ of $K_1(R)\cong {\rm Aut}_{\C_R}({\bf 1}_R)$
 acts on $\phi \in {\rm Mor}_{\C_R}(M,N)$ to give $\alpha\phi: M =
{\bf 1}_R\cdot M \xrightarrow{\alpha\cdot \phi}{\bf 1}_R\cdot N = N$.

\item[B.c)] $\d_R$ preserves the product structure: specifically,
 for each $P$ and $Q$ in $P(R)$ one has $\d_R(P\oplus Q) =
\d_R(P)\cdot\d_R(Q)$.
\item[B.d)]
If $C'\to C\to C''$ is a short exact sequence of complexes,
then there is a canonical isomorphism $\,\d_R(C)\cong \d_R(C')\d_R(C'')$ in $\C_R$ (which we
usually take to be an identification). 
\item[B.e)] If $C$ is acyclic, then the quasi-isomorphism $0\to C$
induces a canonical isomorphism $\,\u_R\to\d_R(C).$
\item[B.f)] For any integer $r$ one has
$\d_R(C[r])=\d_R(C)^{(-1)^r}$.
\item[B.g)] the functor $\d_R$ factorises over the image of
$C^{\rm p}(R)$ in $D^{\rm p}(R)$ and extends (uniquely up to unique isomorphisms) to $(D^{\rm
p}(R),{\rm is}).$ Moreover, if $R$ is regular, also property B.d) extends to all distinguished triangles.
\item[B.h)] For each $C$ in $D^{\rm b}(R)$ we write $\H(C)$ for the
 complex which has $\H(C)^i = H^i(C)$ in each degree $i$ and in which all differentials are $0$. If
 $\H(C)$ belongs to $D^{\rm p}(R)$ (in which case one says that $C$ is {\em cohomologically perfect}), then $C$ belongs to $D^{\rm p}(R)$ and
 there are canonical
 isomorphisms
\[\d_R(C) \cong \d_R(\H(C)) \cong \prod_{i\in \Z} \d_R(H^i(C))^{(-1)^i}.\]
(For an explicit description of the first isomorphism see \cite[\S 3]{km} or \cite[Rem.\
3.2]{Breu-Burns}.)
\item[B.i)] If $R'$ is another (associative unital noetherian) ring and $Y$ an $(R',R)$-bimodule
 that is both finitely generated and projective as an $R'$-module,
 then the functor $Y\otimes_R-: P(R)\to P(R')$ extends to a
commutative diagram
\[\xymatrix{
  (\mathrm{D}^p(R), is) \ar[d]_{Y\otimes_R^\mathbb{L}-} \ar[rr]^{\d_R} & & {\C_R} \ar[d]^{Y\otimes_R-} \\
  (\mathrm{D}^p(R'), is) \ar[rr]^{\d_{R'}} & & {\C_{R'}}   }\]
In particular, if $R\to R'$ is a ring homomorphism and $C$ is in $\mathrm{D}^{\rm p}(R),$ then we
often simply write $\d_R(C)_{R'}$ in place of $R'\otimes_R\d_R(C).$

\item[B.j)]\label{Aj} Let $R^\circ$ be the opposite ring of $R.$ Then the functor $\mathrm{Hom}_R(-,R)$ induces an
anti-equivalence between $\C_R$ and $\C_{R^\circ}$ with quasi-inverse induced  by \linebreak
$\mathrm{Hom}_{R^\circ}(-,R^\circ);$ both functors will be denoted by $-^*.$ This extends to give a
diagram
 \[\xymatrix{
   (\mathrm{D}^p(R), is) \ar[d]_{\mathrm{RHom}_R(-,R)} \ar[rr]^{\d_R}& & {\C_R }\ar[d]^{-^*} \\
   (\mathrm{D}^p(R^\circ),is) \ar[rr]^{\d_{R^\circ}} & & {\C_{R^\circ} } } \]
   which commutes (up to unique isomorphism);   similarly  we have such a commutative diagram for $\mathrm{RHom}_{R^\circ}(-,{R^\circ}).$
\end{itemize}

For the handling of the determinant functor in practice the following considerations are quite
important:
\begin{rem}\label{inverse}
(i) For objects $A,B \in \mathcal{C}_R$ we often identify a morphism $f:A\to B$ with the induced
morphism
\[\xymatrix{
  { \u_R}\ar@^{=}[r]&{A\cdot A^{-1}}\ar[rr]^{f\cdot \id_{A^{-1}}} & &     {B\cdot A^{-1}}.   }
\] Then for morphisms $f:A\to B$ and $g:B\to C$ in $\mathcal{C}_R,$
the composition $ g\circ f: A\to C$  is identified with the product $g\cdot f: \u_R\to C\cdot
A^{-1}$ of $g:\u_R\to C\cdot B^{-1}$ and $f:\u_R\to B\cdot A^{-1}.$ Also, by this identification a
map $f:A\to A$ corresponds uniquely to an element in $K_1(R)=\mathrm{Aut}_{\mathcal{C}_R}(\u_R).$
Furthermore, for a map $f:A\to B$ in $\mathcal{C}_R,$ we write $\overline{f}:B\to A$ for its
inverse with respect to composition, while $f^{-1}=:\overline{\id_{B^{-1}}\cdot f\cdot
\id_{A^{-1}}}:A^{-1}\to B^{-1}$ for its inverse with respect to the multiplication in
$\mathcal{C}_R,$ i.e.\ $f\cdot f^{-1}=\id_{\u_R}.$ Obviously, for a map $f:A\to A$ both inverses
$\overline{f}$ and $f^{-1}$ coincide if all maps are considered as elements of $K_1(R)$ as above.

{\bf Convention:} If $f:\u\to A$ is a morphism and $B$ an object in $\C_R,$ then we write
$\xymatrix{
  B \ar[r]^{\cdot\;f} & B\cdot A   }$ for the morphism $\id_B\cdot f.$ In particular, any morphism $\xymatrix{
  B \ar[r]^{f} & A   }$ can be written as $\xymatrix{
  B \ar[rr]^{\cdot\;(\id_{B^{-1}}\cdot\;f)} &&   A   }.$\\
(ii) The determinant of the complex $C=[P_0\stackrel{\phi }{\to} P_1]$ (in degree $0$ and $1$) with
$P_0=P_1=P$ is by definition $\xymatrix@C=0.5cm{
  { \d_R(C)}\ar@{=}[r]^<(0.3){def} & {\u_R}   }$ and is defined even if $\phi$ is not an
isomorphism (in contrast to $\d_R(\phi)$). But if $\phi$ happens to be an isomorphism, i.e.\ if $C$
is acyclic, then by e) there is also a canonical map $\xymatrix@C=0.5cm{
  {  \u_R}\ar[r]^<(0.3){acyc} & {\d_R(C)}   },$ which is in fact  nothing else then
\[\xymatrix@C=0.5cm{ {\u_R}\ar@{=}[r] & {\d_R(P_1)\d_R(P_1)^{-1}
}\ar[rrr]^{\d(\phi)^{-1}\cdot \id_{\d(P_1)^{-1}}} &&& {\d_R(P_0)\d_R(P_1)^{-1}}   \ar@{=}[r] &
{\d_R(C)} }\] (and  which depends in contrast to the first identification on $\phi$). Hence, the
composite $\xymatrix@C=0.5cm{
  {  \u_R}\ar[r]^<(0.3){acyc} & {\d_R(C)} \ar@{=}[r]^<(0.4){def} & {\u_R}
  }$ corresponds to $\d_R(\phi)^{-1}\in K_1(R)$ according to the first
remark.
 In order to distinguish the above identifications between $\u_R$ and $\d_R(C)$     we also say that $C$ is {\em trivialized by
the identity, } when we refer to $\xymatrix@C=0.5cm{
  { \d_R(C)}\ar@{=}[r]^<(0.4){def} & {\u_R}   }$ (or its inverse with
  respect to composition). For $\phi=\id_P$ both identifications
   agree obviously.
 \end{rem}

We end this section by considering the example where $R=K$ is a field and $V$ a finite dimensional
vector space over $K.$ Then, according to \cite[1.2.4]{fukaya-kato}, $\d_K(V)$ can be identified
with the highest exterior product $\bigwedge^{top}V$ of $V$ and for an automorphism $\phi: V\to V$
the determinant $\d_K(\phi)\in K^\times=K_1(K)$ can be identified with the usual determinant
$\det_K(\phi).$ In particular, we identify $\d_K=K$ with canonical basis $1.$ Then a map
$\xymatrix@C=0.5cm{ {\u_K} \ar[r]^{\psi} & {\u_K}  }$ corresponds uniquely to the value $\psi(1)\in
K^\times.$

\begin{rem}\label{finitemodules}
Note that every {\em finite} $\zp$-module $A$ possesses a free resolution $C$, i.e.\ $\d_{\zp}(A)\cong\d_{\zp}(C)^{-1}=\u_{\zp}.$ Then modulo $\z_p^\times$
the composite $\xymatrix@C=0.5cm{
  {  \u_{\qp}}\ar[r]^<(0.2){acyc} & {\d_{\zp}(C)_{\qp}} \ar@{=}[r]^<(0.4){def} & {\u_{\qp}}
  }$ corresponds to the cardinality $|A|^{-1}\in\Q_p^\times.$
\end{rem}

\include{cmepsfinal.bbl} 

\bibliographystyle{amsplain}

\begin{thebibliography}{10}

\bibitem{benois-berger}
D.\ Benois and L.\ Berger, \emph{Th\'eorie d'{I}wasawa des repr\'esentations cristallines.
              {II}}, Comment. Math. Helv.   \textbf{83} (2008), no.~3, 603--677.

\bibitem{Benois-NQD}
D.~Benois and T.~Nguyen Quang~Do, \emph{Les nombres de {T}amagawa locaux et la
  conjecture de {B}loch et {K}ato pour les motifs {$\Bbb Q(m)$} sur un corps
  ab\'elien}, Ann. Sci. \'Ecole Norm. Sup. (4) \textbf{35} (2002), no.~5,
  641--672.

\bibitem{berger02}
L.~Berger, \emph{Repr\'esentations {$p$}-adiques et \'equations
  diff\'erentielles}, Invent. Math. \textbf{148} (2002), no.~2, 219--284.

\bibitem{berger-exp}
\bysame, \emph{Bloch and {K}ato's exponential map: three explicit formulas},
  Doc. Math. (2003), no.~Extra Vol., 99--129 (electronic), Kazuya Kato's
  fiftieth birthday.

\bibitem{berger-limits}
\bysame, \emph{Limites de repr\'esentations cristallines}, Compos. Math.
  \textbf{140} (2004), no.~6, 1473--1498.

\bibitem{ber} D. ~Bertrand,
\newblock {Valuers de fonctions theta et hauteur p-adiques},
\newblock in {\it S\'eminaire de Th\'eorie des Nombres, Paris,
1980--81}, Progress in Math. {\bf 22}, Birkh\"{a}user, 1982.

\bibitem{bou-ven}
A.~Bouganis and O.~Venjakob, \emph{{On the non-commutative Main Conjecture for
  Elliptic Curves with Complex Multiplication}}, Asian J. Math. \textbf{14}
  (2010), no.~3, 385--416.

\bibitem{Breu-Burns}
M.\ Breuning and D.\ Burns, \emph{Additivity of {E}uler characteristics in
  relative algebraic {$K$}-groups}, Homology, Homotopy Appl. \textbf{7} (2005),
  no.~3, 11--36.
\bibitem{bf}
D.~Burns and M.~Flach, \emph{Tamagawa numbers for motives with (non-commutative)
  coefficients}, Doc. Math. \textbf{6} (2001), 501--570.


\bibitem{bf4}
\bysame, \emph{{On the equivariant Tamagawa number conjecture for
  Tate motives, part II}}, Doc. Math. \textbf{Extra Vol.} (2006), 133--163,
  John H. Coates' Sixtieth Birthday.

\bibitem{burns-ven1}
D.~Burns and O.~Venjakob, \emph{{On the leading terms of zeta isomorphisms and
  p-adic L-functions in non-commutative Iwasawa theory}}, Doc. Math.
  \textbf{Extra Vol.} (2006), 165--209, John H. Coates' Sixtieth Birthday.

\bibitem{burns-ven2}
\bysame, \emph{On descent theory and main conjectures in non-commutative
  {I}wasawa theory}, J. Inst. Math. Jussieu \textbf{10} (2011), no.~1, 59--118.

\bibitem{cher-col1999}
F.~Cherbonnier and P.~Colmez, \emph{Th\'eorie d'{I}wasawa des repr\'esentations
  {$p$}-adiques d'un corps local}, J. Amer. Math. Soc. \textbf{12} (1999),
  no.~1, 241--268.

\bibitem{cfksv}
J.~Coates, T.~Fukaya, K.~Kato, R.~Sujatha, and O.~Venjakob, \emph{{The $GL_2$
  main conjecture for elliptic curves without complex multiplication}}, Publ.\
  Math.\ IHES. \textbf{101} (2005), no.~1, 163 -- 208.

\bibitem{co-su-buch-MC}
J.~Coates and R.~Sujatha, \emph{Cyclotomic fields and zeta values}, Springer
  Monographs in Mathematics, Springer-Verlag, Berlin, 2006.

\bibitem{col1979}
R.~F. Coleman, \emph{Division values in local fields}, Invent. Math.
  \textbf{53} (1979), no.~2, 91--116.

\bibitem{coleman83}
\bysame, \emph{Local units modulo circular units}, Proc. Amer. Math. Soc.
  \textbf{89} (1983), no.~1, 1--7.

\bibitem{colmez-tsinghua}
P.~Colmez, \emph{Fontaine's rings and $p$-adic $l$-functions}, Notes d'un cours
  donn{\'e} {\`a} l'universit{\'e} de Tsinghua en octobre-d{\'e}cembre 2004.

\bibitem{deSh}
E.~de~Shalit, \emph{Iwasawa theory of elliptic curves with complex
  multiplication}, Perspectives in mathematics, vol.~3, Academic Press, 1987.

\bibitem{del-localconstant}
P.~Deligne, \emph{Les constantes des \'equations fonctionnelles des fonctions
  {$L$}}, Modular functions of one variable, {II} ({P}roc. {I}nternat. {S}ummer
  {S}chool, {U}niv. {A}ntwerp, {A}ntwerp, 1972), Springer, Berlin, 1973,
  pp.~501--597. Lecture Notes in Math., Vol. 349.

\bibitem{flach-survey}
M.~Flach, \emph{The equivariant {T}amagawa number conjecture: a survey},
  Stark's conjectures: recent work and new directions, Contemp. Math., vol.
  358, Amer. Math. Soc., Providence, RI, 2004, With an appendix by C.\
  Greither, pp.~79--125.

 \bibitem{flach-survey2}
\bysame, \emph{Iwasawa theory and motivic {$L$}-functions},
  Pure Appl. Math. Q. \textbf{5}, no.~1,
    (2009), 255--294.

\bibitem{fukaya-kato}
T.~Fukaya and K.~Kato, \emph{{A formulation of conjectures on $p$-adic zeta
  functions in non-commutative Iwasawa theory}}, Proceedings of the St.\
  Petersburg Mathematical Society, Vol. XII (Providence, RI), Amer. Math. Soc.
  Transl. Ser. 2, vol. 219, Amer. Math. Soc., 2006, pp.~1--86.

\bibitem{hida93}
H.~Hida, \emph{Elementary theory of {$L$}-functions and {E}isenstein series},
  London Mathematical Society Student Texts, vol.~26, Cambridge University
  Press, Cambridge, 1993.

\bibitem{iz}
D.~Izychev, \emph{{Equivariant $\epsilon$-conjecture for crystalline
  representations}}, PhD thesis, in preparation (2011).

\bibitem{kato-lnm}
K.~Kato, \emph{Lectures on the approach to {I}wasawa theory for {H}asse-{W}eil
  {$L$}-functions via {$B\sb {\rm dR}$}. {I}}, Arithmetic algebraic geometry
  (Trento, 1991), Lecture Notes in Math., vol. 1553, Springer, Berlin, 1993,
  pp.~50--163.

\bibitem{kato-lnmII}
\bysame, \emph{Lectures on the approach to {I}wasawa theory for {H}asse-{W}eil
  {$L$}-functions via {$B\sb {\rm dR}$}. {II}}, preprint (1993).

\bibitem{km}
F.~F. Knudsen and D.~Mumford, \emph{The projectivity of the moduli space of
  stable curves. {I}. {P}reliminaries on ``det'' and ``{D}iv''}, Math. Scand.
  \textbf{39} (1976), no.~1, 19--55.

 \bibitem{loefflerzerbesven}
D.\ Loeffler, O.\ Venjakob and S.\ Zerbes, \emph{{Local epsilon isomorphisms}}, preprint (2013).

\bibitem{zerbes-loeffler}
D.\ Loeffler and S.\ Zerbes, \emph{{Iwasawa theory and $p$-adic $L$-functions
  over $\mathbb{Z}_p^2$-extensions}}, arXiv:1108.5954 (2011).

\bibitem{nek}
J.~Nekov\'a\v r,
\newblock Selmer complexes,
\newblock   Ast\'erisque \textbf{310} (2006).

\bibitem{perrin-height} B.~Perrin-Riou,
\newblock {Th\'eorie d'{I}wasawa et hauteurs {$p$}-adiques},
\newblock  Invent. Math. \textbf{109} (1992), no.\ 1, 137--185.


\bibitem{perrin94}
\bysame, \emph{Th\'eorie d'{I}wasawa des repr\'esentations {$p$}-adiques
  sur un corps local}, Invent. Math. \textbf{115} (1994), no.~1, 81--161, With
  an appendix by Jean-Marc Fontaine.

\bibitem{perrin99}
 \emph{Th\'eorie d'{I}wasawa et loi explicite de r\'eciprocit\'e}, Doc.
  Math. \textbf{4} (1999), 219--273 (electronic).

\bibitem{schmitt}
U.~Schmitt, \emph{{Functorial Properties related to the Main Conjectures of non-commutative Iwasawa Theory}}, PhD thesis, in preparation (2013).

\bibitem{schneider82}
P.~Schneider,
\newblock $p$-adic height pairings. I.,
\newblock Invent. Math. \textbf{69}
  (1982), 401--409.

\bibitem{serre}
J.-P. Serre, \emph{{Abelian l-adic representations and elliptic curves}}, W. A.
  Benjamin, 1968.

 \bibitem{Tate}
J.~Tate, \emph{Number Theoretic Background}, Proceedings of Symposia
in Pure Mathematics,   \textbf{33} (1979).

\bibitem{ven-BSD}
O.~Venjakob, \emph{From the {B}irch and {S}winnerton-{D}yer conjecture to
  non-commutative {I}wasawa theory via the equivariant {T}amagawa number
  conjecture---a survey}, {$L$}-functions and {G}alois representations, London
  Math. Soc. Lecture Note Ser., vol. 320, Cambridge Univ. Press, Cambridge,
  2007, pp.~333--380.

\bibitem{ven-det}
\bysame, \emph{{A note on Determinant functors and Spectral Sequences}},
  {preprint}   ({2012}).

\bibitem{wuth} C.~Wuthrich,
\newblock {On $p$-adic heights in families of elliptic curves},
\newblock  J. London Math. Soc. (2) {\bf 70} (2004), no.~1 , 23--40.

\bibitem{Ya}
R.~I. Yager, \emph{{On two variable p-adic L-functions.}}, Ann. Math., II. Ser.
  \textbf{115} (1982), 411--449.

\bibitem{yasuda}
S.~Yasuda, \emph{Local constants in torsion rings}, J. Math. Sci. Univ. Tokyo \textbf{16} (2009), no.~2, 125--197.





\end{thebibliography}

\def\Dbar{\leavevmode\lower.6ex\hbox to 0pt{\hskip-.23ex \accent"16\hss}D}
  \def\cfac#1{\ifmmode\setbox7\hbox{$\accent"5E#1$}\else
  \setbox7\hbox{\accent"5E#1}\penalty 10000\relax\fi\raise 1\ht7
  \hbox{\lower1.15ex\hbox to 1\wd7{\hss\accent"13\hss}}\penalty 10000
  \hskip-1\wd7\penalty 10000\box7}
  \def\cftil#1{\ifmmode\setbox7\hbox{$\accent"5E#1$}\else
  \setbox7\hbox{\accent"5E#1}\penalty 10000\relax\fi\raise 1\ht7
  \hbox{\lower1.15ex\hbox to 1\wd7{\hss\accent"7E\hss}}\penalty 10000
  \hskip-1\wd7\penalty 10000\box7} \def\Dbar{\leavevmode\lower.6ex\hbox to
  0pt{\hskip-.23ex \accent"16\hss}D}
  \def\cfac#1{\ifmmode\setbox7\hbox{$\accent"5E#1$}\else
  \setbox7\hbox{\accent"5E#1}\penalty 10000\relax\fi\raise 1\ht7
  \hbox{\lower1.15ex\hbox to 1\wd7{\hss\accent"13\hss}}\penalty 10000
  \hskip-1\wd7\penalty 10000\box7}
  \def\cftil#1{\ifmmode\setbox7\hbox{$\accent"5E#1$}\else
  \setbox7\hbox{\accent"5E#1}\penalty 10000\relax\fi\raise 1\ht7
  \hbox{\lower1.15ex\hbox to 1\wd7{\hss\accent"7E\hss}}\penalty 10000
  \hskip-1\wd7\penalty 10000\box7}
\providecommand{\bysame}{\leavevmode\hbox to3em{\hrulefill}\thinspace}
\providecommand{\MR}{\relax\ifhmode\unskip\space\fi MR }
\providecommand{\MRhref}[2]{%
  \href{http://www.ams.org/mathscinet-getitem?mr=#1}{#2}
}
\providecommand{\href}[2]{#2}

\end{document}